\documentclass[11pt]{amsproc}
\usepackage{graphicx, amssymb,amsmath, latexsym,cite,fancyhdr}
\usepackage{amsthm}
\usepackage{times}
\usepackage{tikz}
\usetikzlibrary{arrows,shapes,positioning,calc}
\usepackage[T1]{fontenc} 

\newtheoremstyle{mythm}
{6pt}
{6pt}
{\it}
{}
{\bf}
{.}
{.5em}
{}

\newtheoremstyle{mydef}
{6pt}
{6pt}
{}
{}
{\bf}
{.}
{.5em}
{}

\newtheoremstyle{myrem}
{6pt}
{6pt}
{}
{}
{\bf}
{.}
{.5em}
{}

\theoremstyle{mythm}      
\newtheorem{theorem}{Theorem}
\newtheorem{lemma}[theorem]{Lemma}
\newtheorem{proposition}[theorem]{Proposition}
\newtheorem{corollary}[theorem]{Corollary}

\theoremstyle{mydef}      
\newtheorem{definition}[theorem]{Definition}
\newtheorem{example}[theorem]{Example}

\theoremstyle{myrem}
\newtheorem{remark}[theorem]{Remark}
\numberwithin{equation}{section}

\newcommand{\bullit}{\item[$\bullet$]}

\newcommand{\ad}{{\rm ad}}
\newcommand{\Span}{{\rm Span}}
\newcommand{\SL}{{\mathfrak{sl}}}

\newcommand{\lcm}{{\rm lcm}}
\newcommand{\sign}{{\rm sgn}}
\newcommand{\g}{{\mathfrak g}}
\newcommand{\fa}{{\mathfrak a}}
\newcommand{\fk}{{\mathfrak k}}
\newcommand{\fp}{{\mathfrak p}}
\newcommand{\fs}{{\mathfrak s}}
\newcommand{\ft}{{\mathfrak t}}
\newcommand{\h}{{\mathfrak h}}
\newcommand{\uu}{{\mathfrak u}}

\newcommand{\C}{\mathbb{C}}
\newcommand{\Q}{\mathbb{Q}}

\newcommand{\SqrtField}{\mathbb{Q}^{\sqrt{}}}

\newcommand{\R}{\mathbb{R}}
\newcommand{\Z}{\mathbb{Z}}
\newcommand{\Aut}{{\rm Aut}}

\newcounter{ithmcount}

\newenvironment{items}{
\begin{list}{$\alph{item})$}
{\labelwidth18pt \leftmargin18pt \topsep3pt \itemsep5pt \parsep0pt}}
{\end{list}}

\newenvironment{iprf}{\begin{list}{{\rm
	\alph{ithmcount})}}{\usecounter{ithmcount}\labelwidth-5pt
      \leftmargin0pt \topsep3pt \itemsep1pt \parsep2pt}}{\qedhere\end{list}}

\newenvironment{ithm}{\begin{list}{{\rm \alph{ithmcount})}}{\usecounter{ithmcount}\labelwidth18pt
      \leftmargin18pt \topsep3pt \itemsep1pt \parsep2pt}}{\end{list}}

\allowdisplaybreaks

\begin{document}

\vspace*{-2cm}

\title{A computational approach to the Kostant-Sekiguchi correspondence}
\author[H.\ Dietrich]{Heiko Dietrich}
\author[W.\ A.\ de Graaf]{Willem A.\ de Graaf}
\address{Department of Mathematics, University of Trento, Povo (Trento), Italy}
\email{dietrich@science.unitn.it}
\email{deGraaf@science.unitn.it}
\thanks{This work was supported by a Marie-Curie Fellowship of the European Commission (grant no.\ PIEF-GA-2010-271712).}

\begin{abstract}
Let $\g$ be a real form of a simple complex Lie algebra. Based on  ideas of \DJ okovic and Vinberg, we describe an algorithm to compute representatives 
of the nilpotent orbits of $\g$ using the Kostant-Sekiguchi correspondence. 
Our algorithms are implemented for the computer algebra system {\tt GAP} and, as an application, 
we have built a database of nilpotent orbits of all real forms of simple complex Lie algebras of rank at most 8. 
In addition, we consider two real forms $\g$ and $\g'$ of a complex simple Lie algebra 
$\g^c$ with Cartan decompositions $\g=\fk\oplus\fp$ and $\g'=\fk'\oplus\fp'$. We describe an 
explicit construction of an isomorphism $\g\to\g'$, respecting the given Cartan decompositions, which fails if and only if $\g$ and $\g'$ are not isomorphic. This isomorphism can be used to map the representatives of the nilpotent orbits of $\g$ to other
realisations of the same algebra.
\end{abstract}


\maketitle

\section{Introduction}

\noindent When considering the action of a Lie group on its Lie algebra, the question arises as to 
what its orbits are. This question has mainly been studied for complex simple Lie algebras
$\g^c$, with their adjoint groups $G^c$. Particularly the theory concerning nilpotent
orbits (that is, $G^c$-orbits consisting of nilpotent elements) has seen many interesting developments
over the past decades; we refer to the book of Collingwood \& McGovern \cite{CM} for a detailed account. These orbits have been classified in terms of combinatorial
objects called weighted Dynkin diagrams, using a beautiful connection between nilpotent
orbits and orbits of $\SL_2$-triples. If $\g^c$ is of classical type, then the nilpotent
orbits also have been classified in terms of certain sets of partitions (of the dimension
of the natural representation). 

For real Lie algebras $\g$, with the action of the 
adjoint group $G$, it is much harder to classify the nilpotent ($G$-)orbits. The main problem compared to the complex case is that a weighted Dynkin diagram can correspond to 
several nilpotent orbits. To illustrate this phenomenon consider 
$G^c = \text{PSL}_n(\C)$ and $G= \text{PSL}_n(\R)$ with their Lie algebras 
$\g^c = \SL_n(\C)$ and $\g = \SL_n(\R)$. The nilpotent orbits in $\g^c$ are parametrised
by partitions of $n$, whose parts correspond to the sizes of the Jordan blocks of a
representative of the orbit. The nilpotent orbits in $\g$ are associated with the same partitions, with the difference that the partitions with only even terms 
correspond to two nilpotent orbits.  More generally, the nilpotent orbits of the simple real Lie algebras of classical type have been classified 
in terms of combinatorial objects such as partitions or certain types of Young diagrams, see \cite[Section 9.3]{CM}. For the classification in Lie algebras of exceptional types the main ingredient is the Kostant-Sekiguchi correspondence: Let  $\g=\fk\oplus\fp$ be a Cartan decomposition of the simple real Lie algebra $\g$ 
with complexification $\g^c=\fk^c\oplus \fp^c$. Let $G^c$ be the adjoint group of $\g^c$ and 
denote by $G$, $K$, and $K^c$ the connected Lie subgroups of $G^c$ with corresponding Lie 
algebras $\g$, $\fk$, and $\fk^c$, respectively. The Kostant-Sekiguchi correspondence  states a 
one-to-one correspondence between the nilpotent orbits in $\g$ and the nilpotent $K^c$-orbits 
in $\fp^c$.  Although this correspondence can be described explicitly (as we will do in Section \ref{secKS}), it is difficult to obtain 
concrete representatives of nilpotent orbits in $\g$. Most classification results therefore are 
on the \emph{complex side}, that is, consider nilpotent $K^c$-orbits in $\fp^c$, 
see for example \cite{djokInner, galina, noela, noelb, noel}. However, in tedious work,  \DJ okovic \cite{djokG2, djokE7, djokE8} has used this correspondence  to obtain representatives of the nilpotent orbits for each  of the simple real Lie algebras of exceptional type.

The aim of this paper is to describe methods for constructing  representatives of the nilpotent orbits
of a real simple Lie algebra on a computer. One approach to obtain representatives is to take the existing classifications in the literature, to set up isomorphisms to the algebras
given, and to map the given representatives. However, it is not straightforward to verify the correctness of the representatives given in the literature, 
so this approach is rather error prone. (In fact, in each of his papers \cite{djokG2, djokE7, djokE8} \DJ okovic corrected some errors, due to typos, in his previous papers.) 
For this reason we devise algorithms that effectively carry out the Kostant-Sekiguchi 
correspondence. Since the correctness of each step can be checked algorithmically, we get a
certified list of representatives. 

\subsection{Main results} We describe computational methods to achieve the following three aims:
\begin{items}
\item[A)] Construct isomorphism type representatives for all real forms of a simple complex Lie algebra.
\item[B)] Construct representatives of all nilpotent orbits of a real form constructed in A).
\item[C)] Construct an isomorphism between two given  real forms of a simple complex Lie algebra.
\end{items}

For computational purposes it is often needed that the Lie algebras are given by means of a multiplication table (with respect to some basis). We describe in Section \ref{secPrelimLA} how to construct multiplication tables for all real forms of simple complex Lie algebras (up to isomorphism). 

In Sections \ref{secKS} - \ref{secGetCT} we describe our algorithms to construct representatives of the nilpotent orbits of a Lie algebra constructed in A).  We combine the Kostant-Sekiguchi 
correspondence, see Section \ref{secKS}, with the theory of carrier algebras developed by 
Vinberg \cite{vin87}, see Section \ref{sec:carralg}. This is inspired by \DJ okovic'
proof of the Kostant-Sekiguchi correspondence, see \cite{djok87}.
In Section \ref{secCS} we discuss the construction of so-called Chevalley systems; 
results obtained there will also be important for C). In Section \ref{secGetCT} we discuss the 
main computational problem for applying the Kostant-Sekiguchi correspondence, namely, 
to construct a complex Cayley triple in a $K^c$-orbit of homogeneous 
$\SL_2$-triples; we give more details in Section \ref{secKS}. 

In order to use our lists of representatives of nilpotent orbits also in
other realisations of the Lie algebras (for instance in the split real forms,
in their natural representation), we devise algorithms to construct isomorphisms between 
real simple Lie algebras. More precisely,  in Section \ref{secIso} we discuss the isomorphism problem for two real forms  $\g$ and  $\g'$ of a complex simple Lie algebra $\g^c$. If $\g=\fk\oplus\fp$ and $\g'=\fk'\oplus\fp'$ 
are given Cartan decompositions, then we describe an explicit construction of an isomorphism 
$\g\to\g'$, respecting the given Cartan decompositions, 
which fails if and only if such an isomorphism does not exist. 

\subsection{Related work}
\DJ okovic has first used the Kostant-Sekiguchi correspondence to obtain representatives of nilpotent orbits for the real forms of Lie algebras of exceptional type. His methods used in \cite{djokG2,djokE7,djokE8} vary somewhat
from paper to paper. However, in all these publications the main idea is to start with
a complex nilpotent orbit $\mathcal{O}^c\subset \g^c$ meeting $\g$ non-trivially. Then some real representatives of 
$\mathcal{O}^c$ in $\g$ are computed. The Kostant-Sekiguchi correspondence is used to
decide whether these real representatives lie in the same $G$-orbit or not. The process stops 
when enough elements lying in different $G$-orbits are found. This ad hoc approach has worked for the
Lie algebras of exceptional type, but there is no guarantee that it will always yield 
representatives of all nilpotent orbits. Furthermore, it is rather tedious to apply and 
difficult to translate into a systematic approach suitable for a computer.
  
In our approach the problem is reduced to finding a complex Cayley triple in
a carrier algebra. Most carrier algebras that occur are {\em principal} and for those we have an automatic procedure for finding the triple (see Section  
\ref{secPrincCase}).
However, some carrier algebras are not principal, and for those we translate the problem
into a set of polynomial equations that has to be solved. For dealing with the latter 
problem we use a simple-minded systematic technique (see Section \ref{secNonPrincCase}) which turned out to work well in all our examples, which include all Lie algebras of rank at most 8.

\subsection{Computational remark}
Our algorithms are implemented for the computer algebra system {\tt GAP} \cite{gap}, as
part of a package for doing computations with real Lie algebras, called {\tt CoReLG}. The functions for obtaining the multiplication tables of the real simple Lie algebras
in this package have been implemented by Paolo Faccin; these implementations will be described elsewhere. As an application, we created a database containing representatives of nilpotent orbits 
for all simple real forms of rank at most 8; this database will also be contained in
the package {\tt CoReLG}. As mentioned in the previous paragraph, we construct certain complex Cayley triples in carrier algebras. It is possible that isomorphic carrier algebras will turn up when dealing
with different simple Lie algebras. To avoid dealing with the same problem twice, we have also built a database of non-principal
carrier algebras, together with the Cayley triples that we found (see Section \ref{secDBCA}).

Our approach works uniformly for all simple real Lie algebras. However, our database is currently
limited to the Lie algebras of ranks up to 8 for two reasons. Firstly, it includes all 
exceptional types. Secondly,  in the {\tt SLA} package, the current implementations
of the algorithms for listing the nilpotent orbits of a $\theta$-group are not very efficient
when $\theta$ is an outer automorphism. This makes it currently difficult to go beyond
rank 8 when the real form is defined relative to an outer involution.

There is the question of which base field to use for the computations. The Lie algebras 
that we work with are
defined over $\R$ or $\C$. However, we want to perform exact computations, and the field $\Q$
is not suitable as we often need square roots of rational numbers. For this reason we 
work over the field $\SqrtField=\Q(\{\sqrt{p}\mid p\textrm{ a prime}\})$. 
In Appendix \ref{secSF} we indicate how the arithmetic of that field is implemented.
Since we often work in the complex Lie algebra $\g^c$ in order to obtain results in the
real Lie algebra $\g$, we also use the field $\SqrtField(\imath)$ where $\imath=\sqrt{-1}\in\C$.

\subsection{Notation}\label{secNot}
Throughout this paper we retain the previous notation and denote by $\theta$ the Cartan 
involution associated with the Cartan decomposition  $\g=\fk\oplus\fp$. By $\g^c=\fk^c\oplus 
\fp^c$ we denote the complexification of $\g$, and $\sigma$ is the 
complex conjugation of $\g^c$ with respect to $\g$. 
By abuse of notation, we also denote by $\theta$ its extension to $\g^c$. 
Let $G^c$ be the adjoint group of $\g^c$ and denote by $G$, $K$, and $K^c$ the connected Lie 
subgroups of $G^c$ with corresponding Lie algebras $\g$, $\fk$, and $\fk^c$, respectively.

\section{Constructing the Lie algebras}\label{secPrelimLA}
\noindent The aim of this section is to describe the construction of the real forms we consider. Our computational setup is as in \cite{deGraafBook}, that is, in our algorithms we suppose the Lie algebras are given by multiplication tables, usually with respect to Chevalley bases. For the sake of completeness, we first recall the relevant definitions, and then construct certain bases of all real forms (up to isomorphism) of simple complex Lie algebras.

\subsection{Canonical generators}\label{sec:cons}
Let $\g^c$ be a complex semisimple Lie algebra with Cartan subalgebra $\h^c$. Let  $\Phi$ be the corresponding root system with basis of simple roots $\Delta=\{\alpha_1,\ldots,\alpha_\ell\}$. 
Then $\g^c$ has a {Chevalley basis} with respect to $\Phi$, see \cite[Section 25.2]{humph}:

\begin{definition} A basis $\{h_1,\ldots,h_\ell,x_\alpha\mid \alpha\in\Phi\}$ of $\g^c$ is a 
\emph{Chevalley basis} if  $\{h_1,\ldots,h_\ell\}$ spans the 
Cartan subalgebra $\h^c$ of $\g^c$, and for all $\alpha,\beta\in \Phi$ the following hold:
\begin{items}
\bullit $x_\alpha$ spans the root space $\g_\alpha=\{x\in\g^c\mid \forall i\colon [h_i,x]=\alpha(h_i)x\}$ corresponding to $\alpha$,
\bullit $[x_\alpha,x_{-\alpha}]=h_\alpha$, where $h_\alpha$ is the unique element in $[\g_\alpha,\g_{-\alpha}]$ with $\alpha(h_\alpha)=2$,\newline in particular,  $h_i=h_{\alpha_i}$ for all $i=1,\ldots,\ell$,
\bullit $[x_\alpha,x_\beta]=N_{\alpha,\beta}x_{\alpha+\beta}$ if $\alpha+\beta\in\Phi$, where $N_{\alpha,\beta}\ne 0$ is an integer with $N_{\alpha,\beta}=-N_{-\alpha,-\beta}$,
\bullit $[x_\alpha,x_\beta]=0$ if $\alpha+\beta\notin\Phi$ and $\alpha\ne-\beta$.
\end{items}
\end{definition}

Note that we see the roots in $\Phi$  as elements of the dual space $(\h^c)^*$ via  $[h,x_\alpha] = \alpha(h) x_\alpha$.  For two roots $\alpha,\beta\in \Phi$, the corresponding Cartan integer now is  $\langle \alpha,\beta\rangle=\alpha(h_\beta)$; the Cartan matrix of $\Phi$  defined by $\Delta$ is $(\langle\alpha_i,\alpha_j\rangle)_{ij}$, see \cite[pp.\ 39 \& 55]{humph}. In the sequel, we usually denote by $\{h_1,\ldots,h_\ell,x_\alpha\mid \alpha\in\Phi\}$ a fixed Chevalley basis of $\g^c$, and by $\{h_i,x_i,y_i\mid i=1,\ldots,\ell\}$ with $x_i=x_{\alpha_i}$ and $y_i=x_{-\alpha_i}$ the canonical generating set it contains:

\begin{definition}
A generating set $\{c_i,a_i,b_i\mid i=1,\ldots,\ell\}$ of $\g^c$ is a \emph{canonical generating set} if for all $i,j\in\{1,\ldots,\ell\}$ the following hold:
\begin{items}
\bullit $c_i\in\h^c$, $a_i\in\g_{\alpha_i}$, and $b_i\in\g_{-\alpha_i}$,
\bullit $[c_i,c_j]=0$ and $[a_i,b_j]=\delta_{ij}c_i$, where $\delta_{ij}$ is the Kronecker delta,
\bullit $[c_i,a_j]=\langle \alpha_j,\alpha_i\rangle a_j$ and $[c_i,b_j]=-\langle \alpha_j,\alpha_i\rangle b_j$.
\end{items}
\end{definition}

Let $\{c_i',a_i',b_i'\mid i=1,\ldots,\ell\}$ be a second canonical generating set of $\g^c$, possibly relative to a different basis of simple roots $\Delta'$. If $\Delta$ and $\Delta'$ define the same Cartan matrix, then 
there exists a unique automorphism of $\g^c$ which maps $(c_i,a_i,b_i)$ to 
$(c_i',a_i',b_i')$ for every $i=1,\ldots,\ell$, see \cite[Chapter IV, Theorem 3]{jac}. We freely use this property throughout the paper. Also, if $\Phi$ and $\ell$ follow from the context, then we  write $\{h_i,x_\alpha\mid \alpha,i\}$ and $\{h_i,x_i,y_i\mid i\}$ for the Chevalley basis and canonical generating set.  We end this section with a  proposition, which yields a 
straightforward algorithm to obtain a canonical generating set. For its proof, as well as
the algorithm, we refer to \cite[Section 5.11]{deGraafBook}.

\begin{proposition}\label{prop:cangens}
For $i=1,\ldots,\ell$ let $a_i\in \g_{\alpha_i}$ and $b_i\in \g_{-\alpha_i}$, and write $c_i=[a_i,b_i]$. If  $[c_i,a_i] = 2a_i$ for all $i$, then  $\{c_i,a_i,b_i\mid i\}$ is a canonical generating set of $\g^c$.
\end{proposition}

\subsection{Real forms}\label{secRF}
We now turn to the construction of the real forms of a complex semisimple Lie algebra $\g^c$; without loss of generality, we may assume that $\g^c$ is simple. We continue to use the notation of Section \ref{sec:cons}, that is, $\h^c$ is a Cartan subalgebra of $\g^c$ with root system $\Phi$, 
having basis of simple roots $\Delta=\{\alpha_1,\ldots,\alpha_\ell\}$. Let  $\{h_i,x_\alpha\mid i,\alpha\}$ and $\{h_i,x_i,y_i\mid i\}$ be a corresponding Chevalley basis and canonical generating set.  Recall that a real Lie algebra $\g'$ is a real form of $\g^c$ if $\g^c=\g'\oplus\imath\g'$ as real Lie algebras. \medskip

\subsubsection{Real forms defined by involutions} It is proved in \cite[Theorem 3.1]{onish} that the real subalgebra 
$\uu$ of $\g^c$ defined as \[\uu=\Span_\R(\{ \imath h_1,\ldots,\imath h_\ell, x_\alpha-x_{-\alpha}, 
\imath(x_\alpha+x_{-\alpha}) \mid \alpha\in\Phi^+\})\] is a (compact) real form of $\g^c$. Let $\tau$ be the corresponding real structure, that is, $\tau\colon \g^c\to \g^c$ is the complex conjugation
of $\g^c=\uu\oplus \imath \uu$ with respect to $\uu$. This implies that $\tau(x_\alpha)=-x_{-\alpha}$ for all $\alpha\in\Phi$, in particular, for all $i=1,\ldots,\ell$ we have \[\tau(h_i) = -h_i,\; \tau(x_i) = -y_i,\textrm{ and }\tau(y_i)=-x_i.\] 
It follows from  
\cite[Theorem 3.2]{onish} that, up to isomorphism, every real form of $\g^c$ is constructed as follows:
Let $\theta$ be an involutive automorphism of $\g^c$ commuting with $\tau$. Then  $\uu = \uu_0 \oplus \uu_1$, where $\uu_i$ is the eigenspace of $\theta$ in $\uu$ with eigenvalue $(-1)^i$, and the real form defined by $\uu$ and $\theta$ is \[\g=\g(\theta,\uu) = \fk\oplus \fp\quad\text{with}\quad \fk=\uu_0\text{ and }\fp=\imath\uu_1.\]
This decomposition of $\g$ is a Cartan decomposition whose Cartan involution is the restriction  of $\theta$ to $\g$, cf.\ \cite[\S 5]{onish}. We denote by $\sigma \colon \g^c \to \g^c$ the complex conjugation of $\g^c=\g\oplus \imath\g$ relative to $\g$. 

Two such real forms $\g(\theta,\uu)$ and $\g(\theta',\uu)$ are isomorphic if and only if
$\theta$ and  $\theta'$ are conjugate in $\Aut(\g^c)$. The finite order automorphisms of $\g^c$ are, up to conjugacy, classified by 
so-called Kac diagrams, see \cite[Section 3.3.7]{gorb} or \cite[\S X.5]{helgason}. By running through these diagrams
we can efficiently construct all involutions of $\g^c$ up to conjugacy, and hence all real forms of $\g^c$ up to isomorphism.\medskip

\subsubsection{Real forms of inner type}\label{secRFinn}
Let $\theta$ be an inner involutive automorphism of $\g^c$. Up to conjugacy, $\theta$ maps $(h_i,x_i,y_i)$ to  $(h_i,\lambda_i x_i,\lambda_i^{-1}y_i)$ with $\lambda_i\in\{\pm 1\}$ for all $i$. Clearly, such an automorphism  commutes with $\tau$, and bases of $\fk$ and $\fp$ in $\g=\g(\theta,\uu)=\fk\oplus\fp$ are 
\begin{eqnarray*}
\mathcal{K}&=&\{x_\alpha-x_{-\alpha}, \imath(x_\alpha+x_{-\alpha})\mid \alpha\in\Phi^+
\textrm{ with }\theta(x_\alpha)=x_{\alpha}\}\cup \{\imath h_1,\ldots,\imath h_\ell\},\\
\mathcal{P}&=& \{\imath(x_\alpha-x_{-\alpha}), x_\alpha+x_{-\alpha}\mid \alpha\in\Phi^+
\textrm{ with }\theta(x_\alpha)=-x_{\alpha}\}.
\end{eqnarray*}
 We define $\g$ by the multiplication table constructed via the basis $\mathcal{B} = \mathcal{K}\cup\mathcal{P}$. We note that $\{\imath h_1,\ldots,\imath h_\ell\}$ spans a Cartan subalgebra $\h_0$ of $\fk$, which is also a Cartan subalgebra of $\g$.  It is straightforward to see that $\sigma(x_\alpha) = -x_{-\alpha}$ if $\theta(x_\alpha) = x_\alpha$,
and $\sigma(x_\alpha) = x_{-\alpha}$ if $\theta(x_\alpha) = -x_\alpha$.\medskip

\subsubsection{Real forms of outer type}\label{secRFout}
Let $\theta$ be an outer involutive automorphism of $\g^c$. Up to conjugacy, $\theta = \varphi \circ \chi$,
where $\varphi$ is an involutive diagram automorphism and $\chi$ is an inner involutive automorphism; clearly, $\chi$ and $\varphi$ commute. As above, we can assume that  $\chi$ maps $(h_i,x_i,y_i)$ to $(h_i,\lambda_i x_i,
\lambda_i^{-1}y_i)$ with  $\lambda_i\in\{\pm 1\}$ for all $i$. Further, $\varphi$ maps  $(h_i,x_i,y_i)$ to $(h_{\pi(i)},x_{\pi(i)}, y_{\pi(i)})$ for all $i$, where $\pi$ is an involutive permutation of 
 $\{1,\ldots,\ell\}$ with $(\langle \alpha_j,\alpha_i\rangle)_{ij} = (\langle \alpha_{\pi(j)},\alpha_{\pi(i)}\rangle)_{ij}$; note that  $\{h_{\pi(i)},x_{\pi(i)},y_{\pi(i)}\mid i\}$ also is a canonical generating set, and $\lambda_{\pi(i)} = \lambda_i$ since  $\chi$ and $\varphi$ commute. The permutation $\pi$  induces an automorphism of $\Phi$, which we also denote by $\varphi$; that is, $\varphi(\alpha_i) = \alpha_{\pi(i)}$. 

Let $\g=\g(\theta,\uu)=\fk\oplus\fp$. We now determine bases $\mathcal{K}$ and $\mathcal{P}$ for $\fk$ and $\fp$, and, as before, define $\g$ by the multiplication table constructed via $\mathcal{B} = \mathcal{K}\cup\mathcal{P}$.  Since $\g^c$ admits outer automorphisms, it is of type $A$, $D$, or $E_6$, in particular, simply laced, cf.\ \cite[Table 1]{onish}. We first consider the case where $\g^c$ is not of type $A_{\ell}$ with $\ell$ even. In this case there exists a Chevalley basis 
$\{h_i,\hat x_\alpha\mid i,\alpha\}$ such that, when defining 
$\widehat{N}_{\alpha,\beta}$ by $[\hat x_\alpha, \hat x_\beta ] = \widehat{N}_{\alpha,\beta} \hat 
x_{\alpha+\beta}$, we have $\widehat{N}_{\varphi(\alpha),\varphi(\beta)}=\widehat{N}_{\alpha,\beta}$ for
all $\alpha,\beta\in \Phi$, see \cite[\S 7.9]{kac} or \cite[Section 5.15]{deGraafBook}. (This result does not hold if $\g^c$ is of type $A_\ell$ with $\ell$ even; we consider this case in the following section.)
Induction on the height of $\alpha$ now proves that $\varphi(\hat x_\alpha)= \hat x_{\varphi(\alpha)}$ 
for all $\alpha\in\Phi$. Thus, if $\varphi(\alpha)=\alpha$, then $\varphi$ acts as the identity on $\g_\alpha$, which implies that  $\varphi(x_\alpha) = x_\alpha$. 

For $\alpha\in\Phi$ define  \[ v_\alpha = x_\alpha-\varphi(x_{\alpha})\quad\textrm{and}\quad u_\alpha=\begin{cases} x_\alpha& \textrm{if }\varphi(\alpha)=\alpha,\\ x_\alpha+ \varphi(x_{\alpha})& \textrm{if }\varphi(\alpha)\ne \alpha.\end{cases}\] Let $\Psi^+$ be the set consisting of all $\alpha\in \Phi^+$ such that $\varphi(\alpha)=\alpha$, along
with one element of each pair $(\alpha,\varphi(\alpha))$ where $\varphi(\alpha)\neq \alpha$. Let $\mathcal{I}\subseteq \{1,\ldots,\ell\}$ be a set of representatives of the $\pi$-orbits on 
$\{1,\ldots,\ell\}$ of length 2. Now we define $\mathcal{K}$ as the union of the three sets 
\begin{eqnarray*}
&&\mathcal{H}_0 = \{ \imath h_i \mid i=1,\ldots,\ell\text{ with } \pi(i) = i \} \cup \{ \imath (h_i+h_{\pi(i)})\mid i\in\mathcal{I}\},\\
&&\{ u_\alpha -u_{-\alpha}, \imath(u_\alpha +u_{-\alpha}) \mid \alpha\in\Psi^+\text{ with } \chi(x_\alpha) = x_\alpha \},\textrm{ and }\\
&&\{ v_\alpha -v_{-\alpha}, \imath(v_\alpha +v_{-\alpha}) \mid \alpha\in\Psi^+\text{ with } \chi(x_\alpha) = -x_\alpha \text{ and } \varphi(\alpha)\neq \alpha \};
\end{eqnarray*}
note that if $\varphi(\alpha) = \alpha$ and $\chi(x_\alpha) = x_\alpha$, then $\theta(x_\alpha)
= x_\alpha$, whence $u_\alpha -u_{-\alpha}, \imath(u_\alpha +u_{-\alpha})\in \fk$. 
We define $\mathcal{P}$ to be the union of 
\begin{eqnarray*}
&&\{ h_i-h_{\pi(i)}\mid i\in\mathcal{I}\},\\
&&\{ \imath(u_\alpha -u_{-\alpha}), u_\alpha +u_{-\alpha} \mid \alpha\in\Psi^+\text{ with }  \chi(x_\alpha) = -x_\alpha \}, \textrm{ and }\\
&&\{ \imath(v_\alpha -v_{-\alpha}), v_\alpha +v_{-\alpha} \mid \alpha\in \Psi^+\text{ with } \chi(x_\alpha) = x_\alpha \text{ and }
\varphi(\alpha)\neq \alpha\}.
\end{eqnarray*}
It is straightforward to verify that  $\mathcal{K}$ and $\mathcal{P}$ are bases of $\fk$ and $\fp$.
Further, $\mathcal{H}_0$ spans a Cartan subalgebra  $\h_0$ of 
$\fk$, but this time the complexification $\h_0^c$ is not a Cartan subalgebra of $\g^c$. We have
$\sigma(u_\alpha) = -u_{-\alpha}$ and  $\sigma(v_\alpha) = v_{-\alpha}$ if $\chi(x_\alpha) = x_{\alpha}$, and 
$\sigma(u_\alpha) = u_{-\alpha}$ and $\sigma(v_\alpha) = -v_{-\alpha}$ otherwise.

\begin{remark}
We  consider the weight space decomposition of $\g^c$ with respect to $\h_0^c$ and show that each weight space in $\fk^c$ and $\fp^c$ (corresponding to a non-zero weight) is 1-dimensional. Note that  $\varphi$ fixes $\h_0^c$ pointwise and, if $h\in\h_0^c$, then $\alpha_i(h)\varphi(x_i)=\varphi([h,x_i])=[h,\varphi(x_i)]=\varphi(\alpha_i)(h)\varphi(x_i)$ for all $i$, implying that $\alpha(h)=\varphi(\alpha)(h)$ for all $\alpha\in\Phi$. Now write $\Psi=\Psi^+\cup (-\Psi^+)$ and define $\Psi_0=\{\alpha|_{\h_0^c}\mid \alpha\in\Psi\}$ as a subset of $(\h_0^c)^\star$. Consider the simple Lie algebra $\mathfrak{l} = \{ x\in \g^c \mid \varphi(x) = x \}$, see  \cite[Section 7.9]{kac}. It is easy to
verify that for all $\alpha\in\Psi$ we have $u_\alpha\in\mathfrak{l}$,
and, further, if $h\in\h_0^c$, then $[h,u_\alpha]=\alpha(h)u_\alpha$.
Since $\mathfrak{l}$ is simple, this proves that the root space
decomposition of $\mathfrak{l}$ with respect to $\h_0^c$ is $\mathfrak{l} = 
\h_0^c \oplus \bigoplus_{\alpha\in \Psi} \mathfrak{l}_\alpha$,
where $\mathfrak{l}_\alpha$ is spanned by $u_\alpha$; in particular, $| \Psi|=|\Psi_0|$.
So we have the $\h_0^c$-weight space decompositions
\begin{eqnarray*}\fk^c& =& \h_0^c \oplus \bigoplus\nolimits_{\alpha\in \Psi_0} \fk^c_\alpha\quad\textrm{and}\\
\fp^c &=& \Span_\C(\{h_i-h_{\pi(i)}\mid i=1,\ldots,\ell\text{ with }\pi(i)\neq i\})  \oplus \bigoplus\nolimits_{\alpha\in \Psi_0} \fp^c_\alpha,
\end{eqnarray*}
where each $\fk^c_\alpha = \{ x \in \fk^c \mid \forall h\in \h_0^c:[h,x] = \alpha(h) x\}$ (and similarly $\fp^c_\alpha$) has dimension at most one. More precisely, if $\alpha\in\Phi$ and $\bar\alpha=\alpha|_{\h_0^c}$, then the following hold:
\begin{items}
\bullit if  $\varphi(\alpha)\neq \alpha$ and $\chi(x_\alpha)=x_\alpha$, then  $\fk^c_{\bar\alpha}=\Span_\C(u_\alpha) $ and $\fp^c_{\bar\alpha}=\Span_\C(v_\alpha)$,
\bullit if  $\varphi(\alpha)\neq \alpha$ and $\chi(x_\alpha)\ne x_\alpha$, then  $\fk^c_{\bar\alpha}=\Span_\C(v_\alpha)$ and $\fp^c_{\bar\alpha}=\Span_\C(u_\alpha) $,
\bullit if  $\varphi(\alpha)=\alpha$ and $\chi(x_\alpha)=x_\alpha$, then  $\fk^c_{\bar\alpha}=\Span_\C(u_\alpha)$ and $\fp^c_{\bar\alpha}=0$,
\bullit if  $\varphi(\alpha)= \alpha$ and $\chi(x_\alpha)\ne x_\alpha$, then  $\fk^c_{\bar\alpha}=0$ and $\fp^c_{\bar\alpha}=\Span_\C(u_\alpha) $.
\end{items}
\end{remark}

\medskip

\subsubsection{Real forms of $A_{\ell}$, $\ell$ even, of outer type}\label{secRFout2}
It remains to consider the case where $\g^c$ is of type $A_{\ell}$ with $\ell=2m$ even; we use the notation of the previous section. Up to conjugacy, we can assume that $\chi$ is the identity, thus $\theta=\varphi$ is the unique diagram automorphism.  (This follows directly from looking at
the possible Kac diagrams of an outer involution in this case.)
Since $\g^c$ is simply laced,  $N_{\alpha,\beta}=\pm 1$ for all $\alpha,\beta\in\Phi$ with $\alpha+\beta\in\Phi$, and  induction on the height of $\alpha$ proves that  $\varphi(x_\alpha)=\pm x_{\varphi(\alpha)}$ for all $\alpha\in\Phi$. By \cite[\S 7.10]{kac}, there is a Chevalley basis of $\g^c$ such that $\varphi(x_\alpha) = -x_\alpha$ for all $\alpha\in\Phi$ with $\varphi(\alpha)
= \alpha$. Let $\g=\g(\theta,\uu)=\fk\oplus\fp$. A basis of $\fk$ is the set $\mathcal{K}$ defined as the union of 
\[\mathcal{H}_0 = \{ \imath (h_i+h_{\pi(i)})\mid i\in\mathcal{I}\}\quad\textrm{and}\quad\{ u_\alpha -u_{-\alpha}, \imath(u_\alpha +u_{-\alpha}) \mid \alpha\in\Psi^+ \text{ with } 
\varphi(\alpha)\neq \alpha \};\]
note that $|\mathcal{I}|=m$ since $\pi$ acts fixed-point freely on $\{1,\ldots, 2m\}$. A basis  $\mathcal{P}$  of $\fp$ is the union of 
\begin{eqnarray*}
&&\{ h_i-h_{\pi(i)}\mid i\in\mathcal{I}\},\\
&&\{ \imath(u_\alpha -u_{-\alpha}), u_\alpha +u_{-\alpha} \mid \alpha\in\Psi^+\text{ with }  \varphi(\alpha) 
= \alpha \}, \textrm{ and }\\
&&\{ \imath(v_\alpha -v_{-\alpha}), v_\alpha +v_{-\alpha} \mid \alpha\in \Psi^+\text{ with } 
\varphi(\alpha)\neq \alpha\}.
\end{eqnarray*}
Again, $\mathcal{H}_0$ spans a Cartan subalgebra  $\h_0$ of 
$\fk$, and $\h_0^c$ is not a Cartan subalgebra of $\g^c$. We obtain weight space decompositions of $\fk^c$ and $\fp^c$ as in Section \ref{secRFout}. All non-zero weight spaces with respect to $\h_0^c$ are 1-dimensional and spanned by an $u_\alpha$ or $v_\alpha$. Again, 
$\sigma(u_\alpha) = u_{-\alpha}$ and  $\sigma(v_\alpha) = v_{-\alpha}$.

\section{Kostant-Sekiguchi correspondence}\label{secKS}

\noindent Let $\g^c$ be a complex semisimple Lie algebra with real form $\g=\fk\oplus\fp$,  associated complex conjugation $\sigma$, and Cartan involution $\theta$. Recall that we denote by $G$ and $K^c$ the connected Lie subgroups of the adjoint group $G^c$ of $\g^c$ with Lie algebras $\g$ and $\fk^c$, respectively. The Kostant-Sekiguchi correspondence is a one-to-one correspondence between the nilpotent 
$G$-orbits in $\g$ and the nilpotent $K^c$-orbits  in $\fp^c$. The latter orbits can be 
constructed using the algorithms in \cite{sla,deGraafTheta}; note that $K^c$, together with 
its action on $\fp^c$, is a so-called $\theta$-group. An implementation of the Kostant-Sekiguchi 
correspondence would therefore allow us to construct the nilpotent $G$-orbits in $\g$. 

We now describe this correspondence in more detail. Its proof has been completed independently 
by \DJ okovic \cite{djok87} and Sekiguchi \cite{seki87}; here we follow the description of 
\cite{djok87} and refer to that paper  for an historical account and (references to) proofs. 
First, we need some notation. The following definitions are as in \cite{djok87} with the 
exception that our ``$f$'' has been replaced by ``$-f$''. An \emph{$\SL_2$-triple} in $\g$ (or $\g^c$) is a 
triple $(f,h,e)$ of elements in $\g$ (or $\g^c$) such that $[h,e]=2e$, $[h,f]=-2f$, and $[e,f]=h$. The 
\emph{characteristic} (element) of this triple is $h$.

\begin{definition}\label{defTriples} 
\begin{iprf}
\item An $\SL_2$-triple $(f,h,e)$ in $\g^c$ is \emph{homogeneous} if $e,f\in\fp^c$ and 
$h\in\fk^c$.
\item An $\SL_2$-triple $(f,h,e)$ in $\g^c$ is a \emph{complex Cayley triple} if it is 
homogeneous and $\sigma(e)=f$.
\item An $\SL_2$-triple $(f,h,e)$ in $\g$ is a \emph{real Cayley triple} if $\theta(e)=-f$.
\end{iprf}
\end{definition}

The Kostant-Sekiguchi correspondence can now be stated as in Figure \ref{figKS}, where all maps 
are bijections. We provide some details. Every non-zero nilpotent $e\in\fp^c$ lies in some 
homogeneous $\SL_2$-triple $(f,h,e)$ of $\g^c$, and the projection $(f,h,e)\mapsto e$ induces a 
bijection between the $K^c$-orbits of homogeneous $\SL_2$-triples in $\g^c$ and the $K^c$-orbits 
of non-zero nilpotent elements in $\fp^c$; let $\varphi_1$ denote the inverse of this bijection. 
Every $K^c$-orbit of homogeneous $\SL_2$-triples in $\g^c$ contains a complex Cayley triple and, 
conversely, every $K$-orbit of complex Cayley triples in $\g^c$ is contained in a unique 
$K^c$-orbit of homogeneous $\SL_2$-triples in $\g^c$. Thus, inclusion gives rise to a bijection 
between the $K$-orbits of complex Cayley triples and the $K^c$-orbits of homogeneous 
$\SL_2$-triples in $\g^c$; again, let $\varphi_2$ denote the inverse of this bijection.  Let $(f,h,e)$ be a real Cayley triple. Then its {\em Cayley transform} is the triple
$$(\tfrac{1}{2}(\imath e+\imath f+h),\imath (e-f), \tfrac{1}{2}(-\imath e-\imath f+ h))$$
which is a complex Cayley triple. The inverse Cayley transform maps a complex Cayley
triple $(f,h,e)$ to the real Cayley triple
$$(\tfrac{1}{2}\imath(e-f+ h),e+f, \tfrac{1}{2}\imath(e-f- h)).$$ Taking inverse Cayley transforms induces a bijection $\varphi_3$ between the $K$-orbits of complex 
Cayley triples in $\g^c$ and the $K$-orbits of real Cayley triples in $\g$. The 
projection $(f,h,e)\mapsto e$ yields a bijection $\varphi_4$ between these $K$-orbits of real 
Cayley triples  and the $G$-orbits of non-zero nilpotent elements in $\g$. In conclusion, the Kostant-Sekiguchi correspondence states that $\varphi_4\circ\varphi_3\circ\varphi_2\circ\varphi_1$ is a bijection.

\begin{figure}[htbp]
\begin{center} 
\scalebox{0.8}{
\begin{tikzpicture}[>=latex]

\tikzstyle{edge} = [black,thick,|->];
\tikzstyle{edge2} = [black,thick,->];
\tikzstyle{edgeLabel} = [pos=0.5,text width=23ex, xshift=15ex];
\tikzstyle{edgeLabel2} = [pos=0.5,text width=20ex, xshift=12ex];
\tikzstyle{block} = [rectangle, draw,
          text width=25ex, text centered, rounded corners, minimum height=2em]   
\tikzstyle{form} = [rectangle, text width=10ex, text centered, minimum height=2em]      

\node[block](A){non-zero nilpotent\\ $G$-orbits in $\g$};
\node[form, right=of A](AA){$e^G$};
\node[block,below=of A](B){$K$-orbits of real Cayley triples in $\g$};
\node[form, right=of B](BB){$(f,h,e)^K$};
\node[block, below=of B](C){$K$-orbits of complex Cayley triples in $\g$};
\node[form,right=of C](CC){$(f',h',e')^K$};
\node[block, below=of C](D){$K^c$-orbits of homogeneous $\SL_2$-triples in $\g^c$};
\node[form,right=of D](DD){$(f',h',e')^{K^c}$};
\node[block,below=of D](E){non-zero nilpotent $K^c$-orbits in $\fp^c$};
\node[form,right=of E](EE){$e'^{K^c}$};
\draw ($(BB.north)$) edge[edge] node[edgeLabel]{projection} ($(AA.south)$);
\draw ($(CC.north)$) edge[edge] node[edgeLabel]{inverse Cayley transform}  ($(BB.south)$);
\draw ($(CC.south)$) edge[edge] node[edgeLabel]{inclusion} ($(DD.north)$);
\draw ($(DD.south)$) edge[edge] node[edgeLabel]{projection}  ($(EE.north)$);
\draw ($(E.north)$) edge[edge2] node[edgeLabel2]{$\varphi_1$} ($(D.south)$);
\draw ($(D.north)$) edge[edge2] node[edgeLabel2]{$\varphi_2$} ($(C.south)$);
\draw ($(C.north)$) edge[edge2] node[edgeLabel2]{$\varphi_3$} ($(B.south)$);
\draw ($(B.north)$) edge[edge2] node[edgeLabel2]{$\varphi_4$} ($(A.south)$);
\end{tikzpicture}
}\caption{\label{figKS}
Kostant-Sekiguchi correspondence.}
\end{center}
\end{figure}
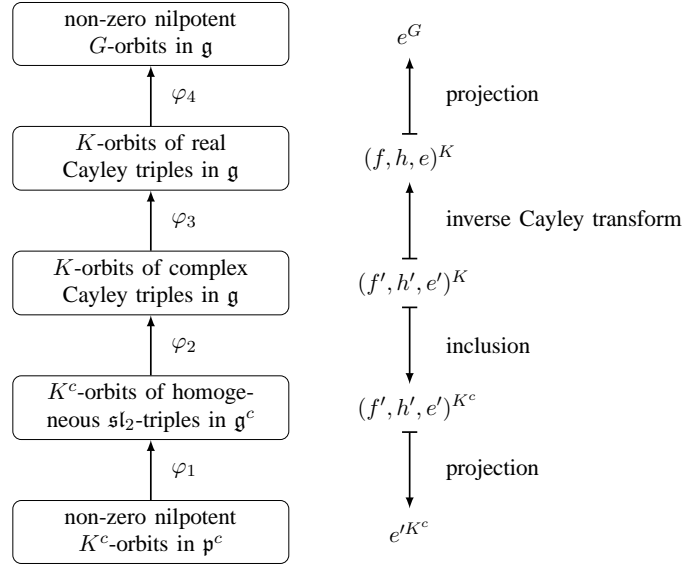

Using the algorithms of  \cite{sla,deGraafTheta}, we can compute all $K^c$-orbits of homogeneous 
$\SL_2$-triples in $\g^c$, which also gives us the bijection $\varphi_1$. A realisation  of the map 
$\varphi_4\circ\varphi_3$ is straightforward. Thus, computationally,  it remains to realise $\varphi_2$, that is:
\medskip

\noindent {\bf Main Problem.} {\it Find a complex Cayley triple in a $K^c$-orbit of homogeneous 
$\SL_2$-triples.}\medskip

We discuss our approach to this problem in Section \ref{secGetCT}. For this purpose, we require some preliminary results;  the subsequent sections therefore introduce carrier algebras and Chevalley systems.

\section{Carrier algebras}\label{sec:carralg}

\noindent We briefly review the theory of carrier algebras as developed by Vinberg \cite{vin87}. In general, carrier algebras are connected to $\Z_m$-graded Lie algebras. Since we exclusively
deal with $\Z_2$-gradings (coming from Cartan decompositions), we only consider this case here.

Let $\g = \fk\oplus \fp$ be as in Section \ref{secRF}, and consider the $\Z_2$-grading
$\g^c = \g_0^c\oplus \g_1^c$, where $\g^c_0 = \fk^c$ and $\g^c_1 = \fp^c$. Recall that  $G_0 = K^c$ is
the connected Lie subgroup of $G^c$ with Lie algebra $\g^c_0$. Let $e\in \g^c_1$ be
nilpotent, and consider the normaliser $N_0(e) = \{ x \in \g^c_0 \mid \exists\lambda\in \C\colon [x,e] = \lambda e\}$. Let $\ft$ be a maximal torus of $N_0(e)$, that is, a maximal abelian subalgebra consisting of semisimple elements, and let $\mu\in \ft^*$ be defined by $[t,e] = \mu(t)e$ for 
$t\in \ft$. Let $\fa^c=\bigoplus_{k\in \Z} \fa_k$ be the $\Z$-graded Lie algebra defined by 
\[\fa_k = \{ x \in \g^c_{k \bmod 2} \mid \forall t\in\ft\colon [t,x] = k\mu(t) x \}.\]
The \emph{carrier algebra} of $e$ is the commutator algebra of $\fa^c$ with the inherited $\Z$-grading, that is, 
\[\fs^c = \fs(e,\ft)=\bigoplus\nolimits_{k\in\Z}\fs_k=[\fa^c,\fa^c].\]  As shown in \cite{vin87}, it has the following
properties:

\begin{items}
\bullit $\fs^c$ is semisimple with  $\dim \fs_0 = \dim \fs_1$,
\bullit $\fs^c$ is not a proper subalgebra of a $\Z$-graded semisimple subalgebra of $\g^c$
of the same rank,
\bullit $\fs_k\subseteq \fk^c$ if $k$ is even, and $\fs_k\subseteq \fp^c$ otherwise,
\bullit $\fs^c$  is normalised by a Cartan subalgebra of $\g^c_0$. 
\end{items}
Moreover, $e\in \fs_1$ is in \emph{general position}, that is, $[\fs_0,e]=\fs_1$; every element in $\fs_1$ in general position is $G_0$-conjugate to $e$. If $(f,h,e)$ is a homogeneous $\SL_2$-triple in $\fs^c$, that is, $h\in\fs_0$ and $f\in\fs_{-1}$, then
$h/2$ is the unique {\em defining element} of $\fs^c$, that is, for all $k$ \[\fs_k = \{ x \in \fs^c \mid  [ \tfrac{h}{2}, x] = kx \}.\]Since all maximal tori of $N_0(e)$ are conjugate, this yields a bijection between the nilpotent $G_0$-orbits in $\g_1^c$ and the $G_0$-conjugacy
classes of $\Z$-graded subalgebras with the above properties. This bijection can be
used for an algorithm to list the nilpotent $G_0$-orbits in $\g^c_1$, cf.\ \cite{deGraafTheta, littelmann}.

\begin{remark}\label{remCAT}
Suppose $(f,h,e)$ is a homogeneous $\SL_2$-triple in $\g^c$ and let $\fs^c=\fs^c(e,\ft)$ be a carrier algebra. Since $h\in N_0(e)$, we can choose a torus containing $h$, thus $h\in\fs^c$. By the Jacobson-Morozov Theorem, see \cite[Theorem X.10.3]{knapp}, there is $f'\in\fs^c$ such that $(f',h,e)$ is an $\SL_2$-triple in $\fs^c$, hence also in $\g^c$. The same theorem shows $f=f'$, thus we can assume that $\fs^c$ contains $f,h,e$. We also call such an $\fs^c$ a carrier algebra of the triple $(f,h,e)$; note that  $h/2$ is its defining element.
\end{remark}

Let $\h_0^c$ be a fixed Cartan subalgebra of $\g_0^c$. A carrier algebra $\fs^c$ is
 {\em standard} if it is normalised by $\h_0^c$, and $[\h_0^c,\fs_k]\subseteq \fs_k$ for all $k$. Since the Cartan subalgebras of
$\g_0^c$ are $G_0$-conjugate, every nilpotent $G_0$-orbit in $\g_1^c$ corresponds to at least
one standard carrier algebra $\fs^c$. Now, as shown in  \cite[p.\ 23]{vin87}, the defining element of $\fs^c$ lies in $\h_0^c$, and $\h_0^c\cap \fs_0$ is a Cartan subalgebra
of $\fs^c$; let $\Phi_{\fs^c}$ be the corresponding root system of $\fs^c$. Clearly, the homogeneous components $\fs_k$ are sums of root spaces, which allows us to define the degree of $\alpha\in\Phi_{\fs^c}$ as  $\deg(\alpha)=k$ if $\fs_\alpha \subseteq \fs_k$. If $\Delta_{\fs^c}$ is a basis of simple roots such that $\deg(\alpha) \geq 0$ for all $\alpha\in\Delta_{\fs^c}$, then in fact $\deg(\alpha)\in \{0,1\}$, see \cite[p.\ 29]{vin87}. If $\deg(\alpha)=1$ for all $\alpha\in\Delta_{\fs^c}$,
then $\fs^c$ is {\em principal}. In that case $\fs_0 = \fs_0 \cap
\h_0^c$ is a torus (in particular, abelian) and $\fs_1$ is spanned by $\fs_\alpha$ with $\alpha\in \Delta_{\fs^c}$.

\section{Chevalley systems}\label{secCS}
\noindent Again, we consider $\g=\fk\oplus \fp$ with Cartan involution $\theta$ and complexification $\g^c$ with complex conjugation $\sigma$. We suppose that $\h^c=\h\oplus\imath\h$ is a Cartan subalgebra of $\g^c$, where $\h$ is a Cartan subalgebra of $\g$ with $\h=(\h\cap\fk)\oplus(\h\cap\fp)$; write $\h_0=\h\cap\fk$ and $\fa=\h\cap\fp$.  With this property, $\h$ is called \emph{standard} and every adjoint $\ad(h)$ with $h\in \h_0$ or $h\in\fa$ has only purely imaginary respectively  only real eigenvalues, see \cite[p.\ 405]{rot72} or \cite[Proposition 5.1(ii)]{onish}.  (This condition on $\h$ is not a serious restriction since every Cartan subalgebra of $\g$ is conjugate to a standard Cartan subalgebra, see \cite[Proposition 1.3]{rot72}.) We let $\Phi$ be the root system of $\g^c$ with respect to 
$\h^c$, with basis of simple roots $\Delta=\{\alpha_1,\ldots,\alpha_\ell\}$.
Further we assume that we have a canonical generating set $\{h_i,x_i,y_i\mid i\}$ such that for every $i$ either $\theta(x_i)=\lambda_i x_i$, with $\lambda_i = \pm 1$, 
or $\theta(x_i)=x_j$ with $i\ne j$. 
We extend these canonical generators to a Chevalley basis $\{h_i, x_\alpha\mid i, \alpha\}$. If $\alpha\in \Phi$ is a root, then $\beta = \alpha\circ\theta$ is a root with $\theta(x_\alpha) \in \g_\beta$ and $\theta(h_\alpha)=h_{\beta}$; hence, by our assumptions, $\Delta$ is stable under $\alpha\mapsto\alpha\circ\theta$. Let $\pi$ be the permutation of $\{1,\ldots,\ell\}$ defined by $\alpha_i\circ\theta=\alpha_{\pi(i)}$. We retain this notation throughout this section.

\begin{lemma}\label{lemSigma} For every $\alpha\in\Phi$ the following hold.
\begin{ithm}
\item $\theta(x_\alpha)=\lambda_\alpha x_{\alpha\circ\theta}$ for some $\lambda_\alpha\in\{\pm1\}$, and $\lambda_\alpha=\lambda_\alpha^{-1}=\lambda_{-\alpha}=\lambda_{\alpha\circ\theta}$.
\item $\sigma(x_\alpha)=r_\alpha x_{-\alpha\circ\theta}$ for some  $r_\alpha\in \R$, and $r_\alpha^{-1}=r_{-\alpha}=r_{-\alpha\circ\theta}$.
\item $\theta(h_\alpha)=h_{\alpha\circ\theta}$ and $\sigma(h_\alpha)=h_{-\alpha\circ\theta}=-h_{\alpha\circ\theta}$.
\end{ithm}
\end{lemma}
\begin{proof}
\begin{iprf}
\item We already know that $\lambda_{\alpha_i}=\lambda_i\in\{\pm 1\}$ and now use induction: If  $\lambda_\alpha,\lambda_\beta\in\{\pm 1\}$, then $N_{\alpha,\beta}\theta(x_{\alpha+\beta})=\lambda_\alpha\lambda_\beta N_{\alpha\circ\theta,\beta\circ\theta} x_{(\alpha+\beta)\circ\theta}$, hence $\lambda_{\alpha+\beta}= N_{\alpha,\beta}^{-1}N_{\alpha\circ\theta,\beta\circ\theta} \lambda_\alpha\lambda_\beta\in\{\pm1\}$ since $|N_{\alpha,\beta}|=|N_{\alpha\circ\theta,\beta\circ\theta}|$; the latter holds since $|N_{\alpha,\beta}|=r+1$ where $r$ is the largest integer with $\alpha-r\beta\in\Phi$, see \cite[Theorem 25.2]{humph}.
Since $\theta$ is an involution, $\lambda_{\alpha\circ\theta}=\lambda_{\alpha}^{-1}$, and $\lambda_{-\alpha}=\lambda_\alpha^{-1}$ follows from $h_{\alpha\circ\theta}=\theta(h_\alpha)=\theta([x_\alpha,x_{-\alpha}])=\lambda_\alpha\lambda_{-\alpha}h_{\alpha\circ\theta}$.

\item[b+c)] Let $\{k_1,\ldots,k_\ell\}$ be a basis of $\h^c=\h_0^c\oplus \fa^c$ such that 
$\{k_1,\ldots,k_m\}$ and $\{k_{m+1},\ldots,k_\ell\}$ form bases of $\fa$ and $\h_0$ respectively.
If $i\in\{1,\ldots,m\}$, then $[k_i,\sigma(x_\alpha)]=\sigma([k_i,x_\alpha]) = \sigma( \alpha(k_i)
x_\alpha) = \alpha(k_i)\sigma(x_\alpha)$ as $\alpha(k_i)$ is real. Analogously, if $i\in\{m+1,\ldots,\ell\}$, then
$\alpha(k_i)$ is purely imaginary and $[k_i,\sigma(x_\alpha)] = -\alpha(k_i)
\sigma(x_\alpha)$.
Hence $\sigma(x_\alpha)=r_\alpha x_{-\alpha\circ\theta}$ with $r_\alpha\in\C$. Note that  $h_{-\beta}=-h_\beta$ for all $\beta\in\Phi$. Now it follows from $[\sigma(h_\alpha),\sigma(x_\alpha)]=2\sigma(x_\alpha)$ that $-\alpha\circ\theta (\sigma(h_\alpha))=2$, 
hence $\sigma(h_\alpha) \in [ \g_{-\alpha\circ\theta},\g_{\alpha\circ\theta}]$ implies that  \[\sigma(h_\alpha) = h_{-\alpha\circ\theta}=- h_{\alpha\circ\theta}.\]
Since $\sigma(h_\alpha)=r_\alpha r_{-\alpha} [x_{-\alpha\circ\theta},x_{\alpha\circ\theta}]=-r_\alpha r_{-\alpha} h_{\alpha\circ\theta}$, this already proves that  $r_\alpha r_{-\alpha}=1$ for all $\alpha\in\Phi$. On the other hand, $r_{\alpha}\overline{r_{-\alpha}}=1$ (with $\overline{\cdot}$ denoting the complex conjugate in $\C$) follows from \[r_{\alpha}\lambda_{-\alpha\circ\theta}x_{-\alpha}=\theta(r_{\alpha}x_{-\alpha\circ\theta})=\theta\circ\sigma(x_\alpha)=\sigma\circ\theta(x_\alpha)=\lambda_\alpha \sigma(x_{\alpha\circ\theta})=\lambda_\alpha\overline{r_{-\alpha}}^{-1}x_{-\alpha};\]recall that $\sigma\circ\theta=\theta\circ\sigma$ and $\lambda_{-\alpha\circ\theta}=\lambda_\alpha$ by a). Together, we have $r_\alpha\in\R$ for all $\alpha\in\Phi$.  Since $\sigma$ has order two, $r_{-\alpha\circ\theta}=r_\alpha^{-1}=r_{-\alpha}$ for all $\alpha\in\Phi$.
\end{iprf}
\end{proof}

As for $\lambda_i=\lambda_{\alpha_i}$, we sometimes write $r_i=r_{\alpha_i}$. We now consider 
Chevalley systems as defined in  \cite[Chapter VIII, \S 3, Definition 3]{bourbaki}, 
see also \cite[Lemma 2]{djok87}.

\begin{definition} We use the previous notation. A \emph{Chevalley system} of $\g^c$ with respect to $\h^c$ is a family $(w_\alpha)_{\alpha\in\Phi}$ 
where $w_\alpha\in \g_\alpha$ with $[w_\alpha,w_{-\alpha}]= -h_\alpha$ for all
  $\alpha\in\Phi$ and such that  the linear map defined by $h\mapsto-h$ for
$h\in\h^c$ and $w_\alpha\mapsto w_{-\alpha}$ for
$\alpha\in\Phi$ is a Lie automorphism, called  \emph{Chevalley automorphism}.  If  $\theta(w_\alpha)=\lambda_\alpha w_{\alpha\circ\theta}$ and $\sigma(w_\alpha)=\lambda_\alpha w_{-\alpha\circ\theta}$ for all $\alpha\in\Phi$ (with $\lambda_\alpha$ as in Lemma \ref{lemSigma}), then $(w_\alpha)_{\alpha\in\Phi}$ is called \emph{adapted} with respect to $\g=\fk\oplus\fp$ (and the Chevalley basis $\{h_i,x_\alpha\mid i,\alpha\}$).
\end{definition}

We first show that adapted Chevalley systems exist. Then, for real forms of inner type, we construct an adapted Chevalley system from our given Chevalley basis, cf.\ \cite[Lemma 2]{djok87}.

\begin{lemma}\label{lemSCSex}
There is an adapted Chevalley system $(v_\alpha)_{\alpha\in\Phi}$ of $\g^c$ with respect to $\g=\fk\oplus\fp$ and $\h^c$.
\end{lemma}
\begin{proof}
For $\alpha\in\Phi$ let $z_\alpha=\varepsilon_\alpha x_\alpha$  
where $\varepsilon_{\alpha}=-1$ if $\alpha\in\Phi^-$ is negative, and
$\varepsilon_\alpha=1$ otherwise. We first prove that $(z_\alpha)_{\alpha\in\Phi}$ is a Chevalley system of $\g^c$.  Clearly, $[z_\alpha,z_{-\alpha}]=\varepsilon_\alpha\varepsilon_{-\alpha}
h_\alpha=-h_\alpha$. Let $\psi$ be the linear map defined by $\psi(h)=-h$ for
$h\in\h^c$ and $\psi(z_\alpha)=z_{-\alpha}$ for
$\alpha\in\Phi$. If $\alpha\in\Phi$, then
$\varepsilon_{-\alpha}=-\varepsilon_\alpha$ and
$\psi(x_\alpha)=\psi(\varepsilon_\alpha
z_\alpha)=\varepsilon_\alpha z_{-\alpha}=-x_{-\alpha}$. If
$\alpha,\beta\in\Phi$ with $\alpha+\beta\in\Phi$, then
$\psi([z_\alpha,z_\beta])=\psi(\varepsilon_\alpha
\varepsilon_\beta N_{\alpha,\beta}x_{\alpha+\beta})=\varepsilon_\alpha
\varepsilon_\beta \varepsilon_{\alpha+\beta}N_{\alpha,\beta}
z_{-\alpha-\beta}$, and $N_{\alpha,\beta}=-N_{-\alpha,-\beta}$ yields
$\psi([z_\alpha,z_\beta])=[\psi(z_{\alpha}),\psi(z_{\beta})]$. Also, 
$\psi([z_\alpha,z_{-\alpha}])=[\psi(z_\alpha),\psi(z_{-\alpha})]$ and 
$\psi([h,z_\alpha])=[\psi(h),\psi(z_\alpha)]$, thus $\psi$ is an automorphism and $(z_\alpha)_{\alpha\in\Phi}$ is a Chevalley system with respect to $\h^c$.

We have seen in Section \ref{secRF} that \[\tilde\uu=\Span_\R(\{\imath h_1,\ldots,\imath h_\ell, x_\alpha-x_{-\alpha},\imath(x_\alpha+x_{-\alpha})\mid \alpha\in\Phi^+\})\] is a compact real form of $\g^c$. If $\tilde\tau$ is the corresponding complex conjugation, then $\tilde\tau(x_\alpha)=-x_{-\alpha}$ for all $\alpha\in\Phi$ and $\tilde\tau(h_i)=-h_i$ for all $i$. In particular, $\tilde\tau$ and $\theta$ commute, and $\tilde\sigma=\theta\circ\tilde\tau$ is a real structure defining a real form $\tilde\g=\g(\theta,\tilde\uu)=\tilde\fk\oplus\tilde\fp$ with Cartan involution $\theta$ (or, more
precisely, the restriction of $\theta$ to $\tilde\g$). If $\pi(i)=i$, then $\imath h_i\in \tilde\fk$, otherwise $\imath(h_i+h_{\pi(i)})\in\tilde\fk$ and $h_i-h_{\pi(i)}\in\tilde\fp$, see Section \ref{secRF}, thus $\tilde\g$ has a standard Cartan subalgebra $\tilde\h$ with $(\tilde\h)^c=\h^c$. It follows readily from the definition
of $\tilde\sigma$ that $\tilde\sigma(z_\alpha)=\lambda_\alpha z_{-\alpha\circ\theta}$ for all $\alpha\in\Phi$. Clearly, $\theta(z_\alpha)=\lambda_\alpha z_{\alpha\circ\theta}$, which shows that $(z_\alpha)_{\alpha\in\Phi}$ is an adapted Chevalley system with respect to $\tilde\g=\tilde\fk\oplus\tilde\fp$ and $\h^c$.

Set $\uu =\fk\oplus \imath \fp$. Then $\uu$ is the compact form of $\g^c$ associated with 
the real form $\g=\fk\oplus\fp$ (cf. Section \ref{secRF}).
Let $\tau \colon \g^c\to \g^c$ be the complex conjugation with respect to $\uu$, then $\sigma = \theta\circ\tau$, and $\theta$ and $\tau$ commute. Thus, $\g$ and $\tilde\g$ both are real forms defined by the automorphism $\theta$ and the compact real structures $\tau$ and $\tilde\tau$, respectively. 
Using Lemma \ref{lemSigma}, we get $\tau(x_i) = r_i\lambda_i y_i$ and  $\tau(y_i) = r_i^{-1}\lambda_i x_i$. Let $\eta\colon\g^c\to \g^c$ be the automorphism which maps $(h_i,x_i,y_i)$ to $(h_i,|r_i|^{-1/2} x_i,|r_i|^{1/2} y_i)$ for all $i$. A short calculation shows that
the compact structures $\eta^{-1}\circ\tau\circ\eta$ and $\tilde\tau$ commute. As shown in 
\cite[Proposition 3.5]{onish}, commuting compact structures are equal, hence $\tau\circ \eta = \eta\circ \tilde\tau$. Again, using Lemma \ref{lemSigma}, we see that $\theta\circ \eta = \eta\circ \theta$, whence also $\sigma\circ\eta = \eta\circ\tilde\sigma$. Now consider $(v_\alpha)_{\alpha\in\Phi}$ with $v_\alpha=\eta(z_{\alpha})$. Clearly, this is a Chevalley system: First, $v_\alpha\in\g_\alpha$ and $[v_\alpha,v_{-\alpha}]=\eta(-h_{\alpha})=-h_\alpha$  for all $\alpha\in\Phi$. Second, if $\psi$ is the Chevalley automorphism corresponding to 
$(z_\alpha)_{\alpha\in\Phi}$, then $\eta\circ\psi\circ\eta^{-1}$ is the Chevalley automorphism
corresponding to  $(v_\alpha)_{\alpha\in\Phi}$. Also, for $\alpha\in\Phi$ we have
$\sigma(v_\alpha)=\sigma\circ\eta(z_\alpha)=\eta\circ\tilde\sigma(z_\alpha)=\lambda_\alpha v_{-\alpha\circ\theta}$ and $\theta(v_\alpha)=\theta\circ\eta(z_\alpha)=\eta\circ\theta(z_\alpha)=\lambda_\alpha v_{\alpha\circ\theta}$, so $(v_\alpha)_{\alpha\in\Phi}$ is adapted with respect to $\g=\fk\oplus\fp$. 
\end{proof}

\begin{proposition}\label{propSpecialCS} We use the previous notation. For $\alpha\in\Phi$ let $z_\alpha=\varepsilon_\alpha x_\alpha$  
where $\varepsilon_{\alpha}=-1$ if $\alpha\in\Phi^-$ is negative, and
$\varepsilon_\alpha=1$ otherwise. Since  $\sigma(z_{\alpha_i})=-r_iz_{-\alpha_{\pi(i)}}$ with  $r_i=r_{\pi(i)}$ by Lemma \ref{lemSigma}, there are $\tau_i=\tau_{\pi(i)}
\in\R$ such that $\tau_i\tau_{\pi(i)}r_i\in\{\pm1\}$.  Let $\psi$ be the automorphism of $\g^c$ mapping 
$(h_i,x_i,y_i)$ to $(h_i,\tau_ix_i,\tau_i^{-1}y_i)$ for all $i$; then $\psi$ commutes with $\theta$. Define $w_\alpha=\psi(z_\alpha)$ for $\alpha\in\Phi$.
\begin{ithm}
\item $(z_\alpha)_{\alpha\in\Phi}$ and  $(w_\alpha)_{\alpha\in\Phi}$ are Chevalley systems with respect to $\h^c$.
\item $\theta(w_{\alpha_i})=\lambda_i w_{\alpha_{\pi(i)}}$ and $\sigma(w_{\alpha_i})=\lambda_i w_{-\alpha_{\pi(i)}}$ for all $i$.
\item If $\g$ is of inner type, then $(w_\alpha)_{\alpha\in\Phi}$ is adapted with respect to $\g=\fk\oplus\fp$.
\end{ithm}
\end{proposition}
\begin{proof}
\begin{iprf}
\item This follows as in the  proof of Lemma \ref{lemSCSex}.
\item Recall $z_{\alpha_i}=x_i$ 
and $z_{-\alpha_{\pi(i)}}=-y_{\pi(i)}$ for all $i$. Now  
$w_{-\alpha_{\pi(i)}}=\psi(z_{-\alpha_{\pi(i)}})=\tau_{\pi(i)}^{-1}z_{-\alpha_{\pi(i)}}$ yields \[\sigma(w_{\alpha_i})=
\sigma(\psi(z_{\alpha_i}))=\sigma
(\tau_i z_{\alpha_i})=-\tau_ir_i z_{-\alpha_{\pi(i)}}=-\tau_ir_i\tau_{\pi(i)}w_{-\alpha_{\pi(i)}}=r_i'w_{-\alpha_{\pi(i)}},\] 
where $r_i'=-\tau_i\tau_{\pi(i)}r_i\in\{\pm1\}$. We have $\theta(x_i)=\lambda_i x_i$ if $\pi(i)=i$, and $\theta(x_i)=x_{\pi(i)}$ otherwise, and, therefore, $\tau_i=\tau_{\pi(i)}$ implies that  $\theta(w_{\alpha_i})=\lambda_i w_{\alpha_i}$ if $\pi(i)=i$, and $\theta(w_{\alpha_i})=w_{\alpha_{\pi(i)}}$ otherwise. By Lemma \ref{lemSCSex}, there exists an adapted Chevalley 
system $(v_\alpha)_{\alpha\in\Phi}$ with respect to $\g=\fk\oplus\fp$ and $\h^c$; each $v_\alpha$ can be written as  
$v_\alpha=c_\alpha w_\alpha$ for some  $c_\alpha\in\C$. It follows from 
\[-h_i=[v_{\alpha_i},v_{-\alpha_i}]=c_{\alpha_i}c_{-\alpha_i}[w_{\alpha_i},w_{-\alpha_i}]=
-c_{\alpha_i}c_{-\alpha_i}h_i\] that $c_{-\alpha_i}=c_{\alpha_{i}}^{-1}$ for all $i$. If $\pi(i)\ne i$, then   
\[c_{\alpha_i}w_{\alpha_{\pi(i)}}=\theta(c_{\alpha_i}w_{\alpha_i})=\theta(v_{\alpha_i})=v_{\alpha_{\pi(i)}}=c_{\alpha_{\pi(i)}}w_{\alpha_{\pi(i)}},\] hence $c_{\alpha_i}=c_{\alpha_{\pi(i)}}$ for all $i$. Thus $\overline{c_{\alpha_i}}c_{\alpha_{\pi(i)}}>0$ is real for every $i$, and  $r_i'=\lambda_{i}$ follows from $r_i',\lambda_i\in\{\pm1\}$ and \[\lambda_iv_{-\alpha_{\pi(i)}}=\sigma(v_{\alpha_i})=\overline{c_{\alpha_i}}\sigma(w_{\alpha_i})
=\overline{c_{\alpha_i}}r_i'c_{-\alpha_{\pi(i)}}^{-1}v_{-\alpha_{\pi(i)}}=
r_i'\overline{c_{\alpha_i}}c_{\alpha_{\pi(i)}}v_{-\alpha_{\pi(i)}}.\]

\item  By b) we know that  $\sigma(w_{\alpha_i})=\lambda_{\alpha_i} w_{-\alpha_{\pi(i)}}$, and Lemma \ref{lemSigma} yields $\sigma(w_{-\alpha_i})=\lambda_{-\alpha_i} w_{\alpha_{\pi(i)}}$ for $i=1,\ldots,\ell$. For $\alpha,\beta\in \Phi$ with $\alpha+\beta\in\Phi$ write
$[w_\alpha,w_\beta]=M_{\alpha,\beta}w_{\alpha+\beta}$ where  $M_{\alpha,\beta}=M_{-\alpha,-\beta}$ is real 
(in fact, integral). Suppose now that for $\alpha,\beta\in\Phi$ we have $\sigma(w_\alpha)=
\lambda_{\alpha} w_{-\alpha\circ\theta }$ and
$\sigma(w_\beta)=\lambda_{\beta} w_{-\beta\circ\theta}$. Then
\[M_{\alpha,\beta}\sigma(w_{\alpha+\beta})=\sigma([w_\alpha,w_\beta])=
[\sigma(w_\alpha),\sigma(w_\beta)]=\lambda_\alpha
\lambda_\beta M_{-\alpha\circ\theta,-\beta\circ\theta } w_{-(\alpha+\beta)\circ\theta}.\] 
If $\g$ is of inner type, then $\alpha\circ\theta=\alpha$ and $\lambda_\alpha\lambda_\beta=\lambda_{\alpha+\beta}$ for all $\alpha,\beta \in\Phi$. Thus, in this case, $M_{\alpha,\beta}= M_{-\alpha\circ\theta,-\beta\circ\theta }= M_{-\alpha,-\beta }$, and induction on the height of $\alpha$ proves that $\sigma(w_\alpha)=\lambda_\alpha w_{-\alpha\circ\theta}$. Similarly,  $\theta(w_\alpha)=\lambda_\alpha w_{\alpha\circ\theta}$ for all $\alpha$, thus $(w_\alpha)_{\alpha\in\Phi}$ is adapted with respect to $\g=\fk\oplus \fp$.
\end{iprf}
\end{proof}

The proof of Proposition \ref{propSpecialCS}b) has the following important corollary, which we use  in Section \ref{secIso}. Recall that  $\theta(x_i)=\lambda_i x_{\pi(i)}$ and $\sigma(x_i)=r_i y_{\pi(i)}$ for all $i$.

\begin{corollary}\label{corSign}
The coefficients $r_i$ and $-\lambda_i$ have the same sign for all $i$.
\end{corollary}
\begin{proof}
In the proof of Proposition \ref{propSpecialCS}b) we have shown that $\lambda_i = r_i'=-\tau_i \tau_{\pi(i)} r_i = -\tau_i^2 r_i$.
\end{proof}

\section{Constructing complex Cayley triples}\label{secGetCT}

\noindent Let $\g=\fk\oplus\fp$  be as  in Section \ref{secRF}, with complexification $\g^c$, Cartan involution $\theta$, and complex conjugation $\sigma$. As usual, we denote by $\Phi$ a root system of $\g^c$ with basis of simple roots $\Delta$; let $\{h_i,x_\alpha\mid i,\alpha\}$ be a corresponding Chevalley basis. We now discuss our Main Problem, cf.\ Section \ref{secKS}, that is, given a homogeneous $\SL_2$-triple $(f,h,e)$ in $\g^c$, we want to construct a complex Cayley triple 
$(f',h',e')$ which is $K^c$-conjugate to $(f,h,e)$.  As constructed in Section \ref{sec:carralg}, we also assume we have a
standard carrier algebra ${\fs^c}=\fs^c(e,\mathfrak{t})$ containing $f,h,e$, cf.\ Remark \ref{remCAT}, and normalised by the Cartan subalgebra $\h_0^c = \h_0 +\imath
\h_0$ of $\fk^c$ with  $\h_0\subseteq \fk$ as in Section \ref{secRF}. 

We will see in Section  \ref{secCSCT} that ${\fs^c}$ is $\sigma$-stable, hence $\fs={\fs^c}\cap \g$ is a real form of ${\fs^c}$. Also, we will see that ${\fs^c}$ is $\theta$-stable, thus  \[(\star)\quad \fs = ({\fs^c}\cap\fk) \oplus ({\fs^c}\cap\fp)\] is a Cartan decomposition whose Cartan involution is the restriction of $\theta$ to $\fs$.  Note that $\fs_0\cap \fk^c$ and $\fs_0\cap \fk$ contain Cartan subalgebras of ${\fs^c}$ and ${\fs}$, respectively, namely, 
$\h_0^c\cap \fs_0$ and $\h_0\cap \fs_0$. In particular, the real form ${\fs}$ is always of inner type and $\h_0\cap \fs_0$ is a standard Cartan subalgebra. Thus the results of Section \ref{secCS} can be applied: we show in Section \ref{secCSCT} how to construct an adapted Chevalley system for ${\fs^c}$; here adapted always means with respect to $\h_0^c\cap \fs_0$, the Cartan decomposition
$(\star)$, and a chosen Chevalley basis of $\fs^c$.

By construction, the triple $(f,h,e)$ is also a homogeneous $\SL_2$-triple in ${\fs^c}$. The approach of \cite{djok87} 
is to find $x\in\fs_1$ with $[x,\sigma(x)]=h$ so that $(\sigma(x),h,x)$ is a complex Cayley 
triple in ${\fs^c}$, thus also in $\g^c$. By the Kostant-Sekiguchi correspondence and 
\cite[Lemma 4]{rallis}, such an $x$ exists and $(\sigma(x),h,x)$ is $K^c$-conjugate to 
$(f,h,e)$. If ${\fs^c}$ is principal, then Chevalley systems can be used to find $x$, see Section \ref{secPrincCase}. If ${\fs^c}$ is not principal, then we make a case distinction 
and use induction, see Section \ref{secNonPrincCase}.

\subsection{Constructing an adapted Chevalley system}\label{secCSCT}
In the following, let $\Phi_{\fs^c}$ be the root system of ${\fs^c}$ with respect to $\h_0^c\cap \fs_0$; let $\Delta_{\fs^c}=\{\beta_1,\ldots,\beta_s\}$ be a basis of simple roots. As mentioned in Section \ref{sec:carralg}, we can assume that each root space $\fs_{\beta_i}$ either lies in $\fs_0$ or $\fs_1$.\medskip

\subsubsection{Inner type} If $\g$ is of inner type, then $\h_0^c\leq \fk^c$ is also a Cartan subalgebra of $\g^c$, hence 
$\Phi_{\fs^c}$ can be considered as a root subsystem of $\Phi$. This implies that  $\{x_\alpha\mid \alpha\in \Phi_{\fs^c}\}$, along with
certain elements of $\h_0^c\cap {\fs_0}$, forms a Chevalley basis of ${\fs^c}$. We denote it by $\{k_i,w_\alpha\mid \alpha\in\Phi_{\fs^c},  i=1,\ldots, s\}$ and let $\{k_i,a_i,b_i\mid i\}$ be the canonical generating set it contains. As usual, write $k_\alpha=[w_\alpha,w_{-\alpha}]$ for  $\alpha\in \Phi_{\fs^c}$. We have seen in Section \ref{secRFinn} that $\sigma(w_\alpha) \in\{ \pm w_{-\alpha}\}$ and $\theta(w_\alpha)\in\{\pm w_\alpha\}$,  hence ${\fs^c}$ is $\sigma$- and $\theta$-stable.
For $\Phi_{\fs^c}$ we use an ordering compatible with that of $\Phi$. Let $z_\alpha =
w_\alpha$ and $z_{-\alpha}=-w_{-\alpha}$ for $\alpha\in\Phi_{\fs^c}^+$, then  $(z_\alpha)_{\alpha\in\Phi_{\fs^c}}$ is
an adapted Chevalley system of ${\fs^c}$.\medskip

\subsubsection{Outer type} Now let  $\g$ be of outer type with defining outer automorphism $\theta=\varphi\circ
\chi$. By construction, each homogeneous
component ${\fs}_k$ lies either in $\fk^c$ or in $\fp^c$, which shows that ${\fs^c}$ is $\theta$-stable. By definition, each $\fs_k$ is normalised
by $\h_0^c$, thus it is a sum of weight spaces (with respect to $\h_0^c$) as considered in  Sections \ref{secRFout} and \ref{secRFout2}; in the following we use the notation introduced in these sections. Let $\alpha\in\Phi_{\fs^c}$, then $\fs_\alpha$ is an $\h_0^c$-weight space, and it is either
contained in $\fk^c$ or in $\fp^c$ (since it lies in a homogeneous component $\fs_k$). These
observations show that there is an $\alpha'\in \Phi$ such that  either 
$\fs_{\alpha}=\Span_\C(u_{\alpha'})$ or $\fs_{\alpha}=\Span_\C(v_{\alpha'})$, and, accordingly,  $\fs_{-\alpha}=\Span_\C(u_{-\alpha'})$ or $\fs_{-\alpha}=\Span_\C(v_{-\alpha'})$.  Since $\sigma(u_{\alpha'})=\pm u_{-\alpha'}$ and $\sigma(v_{\alpha})=\pm v_{-\alpha'}$, this shows that ${\fs^c}$ is stable under $\sigma$. We can now define a new set of canonical generators $\{k_i,a_i,b_i \mid i=1,\ldots,s\}$ for $\fs$; we make a case distinction:
\begin{items}
\bullit If $\fs_{\beta_i}$ is spanned by $u_\alpha=x_\alpha$ with $\varphi(\alpha) = \alpha$, then define
$a_i = x_\alpha$, $b_i = x_{-\alpha}$ and $k_i = [a_i,b_i]$. 
\bullit Now let $\fs_{\beta_i}$ be spanned by $u_\alpha= x_\alpha+x_{\varphi(\alpha)}$ with 
$\varphi(\alpha) \neq \alpha$. Note that $\beta=\alpha-\varphi(\alpha)$ is not a root because  $\varphi$ maps positive roots on positive roots but $\varphi(\beta)=-\beta$. This proves 
 $[u_\alpha,u_{-\alpha}] = h_\alpha + h_{\varphi(\alpha)}$. Also, it follows that $\langle \alpha,\varphi(\alpha)\rangle\leq 0$, cf.\ \cite[Lemma 9.4]{humph}, and finally $\langle \alpha,\varphi(\alpha)\rangle \in \{0,-1\}$, as $\Phi$ is simply laced, which means that there is only one root length; in particular, $\langle \alpha,\varphi(\alpha)\rangle = \langle\varphi(\alpha),\alpha\rangle$. The latter now implies that $[h_\alpha + h_{\varphi(\alpha)},u_\alpha] = (2+\langle \varphi(\alpha),\alpha\rangle)u_\alpha$ since $\varphi(\alpha)(h_\alpha)=\langle \varphi(\alpha),\alpha\rangle = \langle \alpha,\varphi(\alpha)\rangle=\alpha(h_{\varphi(\alpha)})$. If $\langle \varphi(\alpha),\alpha\rangle=0$, then we define $a_i = u_\alpha$, $b_i = 
u_{-\alpha}$, and $k_i = [a_i, b_i]$. Otherwise,  we set $a_i = \sqrt{2} u_\alpha$, $b_i = 
\sqrt{2} u_{-\alpha}$, and $k_i = [a_i, b_i]$.
\bullit If $\fs_{\beta_i}$ is spanned by $v_\alpha= x_\alpha-x_{\varphi(\alpha)}$, then we do exactly 
the same  as in the previous case with $u$ replaced by $v$.
\end{items}

In all cases we find $a_i\in \fs_{\beta_i}$, $b_i \in \fs_{-\beta_i}$, and 
$k_i = [a_i, b_i ]$ such that $[k_i ,a_i] = 2b_i$ for all $i$. By Proposition \ref{prop:cangens},  $\{k_i, a_i, b_i\mid i\}$ is a canonical generating set for $\fs^c$, and, by construction, $\sigma(a_i)= \pm b_i$ for all $i$. We extend this canonical generating set to a Chevalley basis $\{k_i, w_\alpha\mid i,\alpha\}$ of  ${\fs^c}$  such that $w_{\alpha_i} = a_i$ and $w_{-\alpha_i} = b_i$; as usual, write $k_\alpha=[w_\alpha,w_{-\alpha}]$ for all $\alpha$. We now define $z_\alpha = w_\alpha$ for $\alpha >0$ and $z_\alpha =
-w_{\alpha}$ for $\alpha <0$; it is straightforward to verify that $(z_{\alpha})_{\alpha\in\Phi_{\fs^c}}$ 
is an adapted Chevalley system  of ${\fs^c}$, cf.\  Proposition \ref{propSpecialCS}.\medskip

The conclusion is that for all $\g$ we can find an adapted Chevalley system of ${\fs^c}$, and the coefficients 
of its elements with respect to the given basis  of $\g$ lie in 
$\Q(\imath, \sqrt{2})$, in particular, in $\SqrtField(\imath)$.

\subsection{The principal case}\label{secPrincCase} 

This construction follows \cite[Lemma 3]{djok87}. We use the previous notation and suppose 
that the carrier algebra  ${\fs^c}$ of  $(f,h,e)$ is principal, that is, there is a basis 
$\Delta_{\fs^c}$ of $\Phi_{\fs^c}$ such that for every 
$\alpha\in\Delta_{\fs^c}$ we have $\fs_\alpha\subseteq \fs_1$. Let 
$(z_\alpha)_{\alpha\in\Phi_{\fs^c}}$ be the adapted Chevalley system for ${\fs^c}$ as constructed in the previous section. We want to find $x\in \fs_1$ with $[x,\sigma(x)]=h$ 
of the form  $x=\sum_{\alpha\in\Delta_{\fs^c}} c_\alpha z_\alpha$ with all $c_\alpha$ real. Note that 
$\sigma(x)=-\sum_{\alpha\in\Delta_{\fs^c}} c_\alpha z_{-\alpha}$ and $\alpha-\beta\notin\Phi_{\fs^c}$ for all 
$\alpha,\beta\in\Delta_{\fs^c}$. Thus, the equation we have to solve is $h=[x,\sigma(x)]=
\sum_{\alpha\in\Delta_{\fs^c}} c_\alpha^2 k_\alpha$; recall that $[z_\alpha,z_{-\alpha}]= -k_\alpha$. Note 
that $\beta(h)=2$ for all $\beta\in\Delta_{\fs^c}$ since $z_\beta\in\fs_1$ and $h/2$ is the defining 
element of ${\fs^c}$. This shows that our equation is equivalent to the system of equations  
$2=\sum_{\alpha\in\Delta_{\fs^c}} d_\alpha \beta(k_\alpha)$ with $d_\alpha=c_\alpha^2$ where $\beta$ ranges 
over $\Delta_{\fs^c}$.  The coefficients $\beta(k_\alpha)$ of this system are the entries 
of the Cartan matrix of $\Phi_{\fs^c}$, whose inverse has non-negative
entries, cf.\ \cite[Section 13.1]{humph}. Thus, the system  has a solution with all $d_\alpha\geq 0$
real. In conclusion, to construct $x$, we 
first compute $h=\sum_{\alpha\in\Delta_{\fs^c}} d_\alpha k_\alpha$, and then set $x=\sum_{\alpha\in\Delta_{\fs^c}} 
c_\alpha z_\alpha$ where $c_\alpha=\sqrt{d_\alpha}$ is real for every $\alpha\in\Delta_{\fs^c}$. We now 
show that each $c_\alpha\in \SqrtField$. For every $\alpha\in\Delta_{\fs^c}\subseteq \Phi_{\fs^c}$, the 
element $k_\alpha=[w_\alpha,w_{-\alpha}]$ is a $\Z$-linear combination of $k_1,\ldots,k_s$, the 
elements of the Chevalley basis of $\fs^c$ that span its Cartan subalgebra $\h_0^c\cap \fs_0$, see \cite[Theorem 25.2]{humph}. As shown in the previous paragraph, these elements are 
$\Z$-linear combinations of  $h_1,\ldots,h_\ell$, the elements elements of the Chevalley basis 
of $\g^c$ that span $\h^c$. Similarly, the element $h$, which is the characteristic of an
$\SL_2$-triple, is a $\Z$-linear combination of $h_1,\ldots,h_\ell$. 
 Together, all this implies that the 
 $d_\alpha$ are in fact rational, thus $c_\alpha\in\SqrtField$.

\subsection{Non-principal case}\label{secNonPrincCase} 

Now suppose that the carrier algebra ${\fs^c}$ of $(f,h,e$) is non-principal. As mentioned above, there exists $x\in \fs_1$ such that $(\sigma(x),h,x)$ is a complex Cayley 
triple in the same $K^c$-orbit as $(f,h,e)$. However, constructing $x$ is not straightforward.
We first set up the system of rational polynomial equations in the coefficients of $x$ with 
respect to a basis of $\fs_1$, equivalent to $[x,\sigma(x)]=h$. Note that this is a 
system of $\dim \fs_0$ polynomial equations in $\dim \fs_1$ variables.
Then in order to solve them we use a brute-force approach, that is, for $i=1,2,3,\ldots$
we set all but $i$ indeterminates in these equations to zero. For each equation system that
arises we check, using Gr\"obner bases (see for example \cite{clo}),
whether a solution over $\C$ exists. We stop when we 
find an equation system consisting of equations of the form $T^2 = a$, where $a\in \Q$ and
$T$ is an indeterminate, or $T_c=a_{c_1}T_{c_1}^2+\ldots+a_{c_m}T_{c_m}^2$, where each $T_{c_i}$ satisfies an equation of the first type. It is then straightforward to obtain a solution over  
$\SqrtField(\imath)$. This systematic approach for constructing a complex Cayley triple $(\sigma(x),h,x)$ can easily be carried out automatically by a computer. It turned out to work well in all our computations for the carrier algebras in the real forms constructed in Section \ref{secRF}; our experiments include all simple real Lie algebras of rank at most 8. Unfortunately, we have no proof that a solution of the equation system always exists over the field $\SqrtField(\imath)$, hence we cannot prove that our approach will always work.\medskip


\subsubsection{A database}\label{secDBCA}
To reduce work, we have constructed a database of the simple non-principal 
carrier algebras that appeared during our calculations. Let ${\fs^c}$ be such a carrier algebra.  As shown in Section \ref{secCSCT}, there is a canonical generating set $\{k_i,a_i,b_i\mid i=1,\ldots,s\}$ of $\fs^c$ such that $a_i\in\fs_{\varepsilon_i}$ with $\varepsilon_i\in\{0,1\}$ and $\sigma(a_i)=\lambda_i b_i$ with $\lambda_i\in\{\pm 1\}$ for all $i$. Since $\sigma(k_i) = -k_i$ for all $i$, the map $\sigma$ is determined  by the signs $\lambda_1,\ldots,\lambda_s$. Moreover,  $k_1,\ldots,k_s\in
\fs_0$, and if $a_i\in\fs_k$, then $b_i\in \fs_{-k}$. Thus, the following data describes ${\fs^c}$, its grading, and $\sigma$ completely; we store this data in our database:

\begin{items}
\bullit a multiplication table, canonical generators $\{k_i,a_i,b_i\mid i\}$, and Cartan matrix $C$,
\bullit the signs $\lambda_1,\ldots,\lambda_s$ and $\varepsilon_1,\ldots,\varepsilon_s$,
\bullit a complex Cayley triple $(f,h,e)$ in $\fs$ such that $e\in \fs_1$ is in general position
\end{items}

Suppose in our computations we consider a real semisimple Lie algebra $\g' = \fk'\oplus \fp'$ with complexification
$(\g')^c = (\fk')^c \oplus (\fp')^c$ and complex conjugation $\sigma'$. Let $(f',h',e')$ be a homogeneous $\SL_2$-triple in $(\g')^c$, and we want to find a conjugate complex Cayley triple in $(\g')^c$. As before, we first construct the carrier algebra ${(\fs')^c}$ of the triple. If it is principal, then we proceed as in Section \ref{secPrincCase}, so let it be non-principal. Recall that ${(\fs')^c}$ is semisimple and, by considering its simple components separately, we can assume that ${(\fs')^c}$ itself is simple. Suppose in our database there exists a simple carrier algebra ${\fs^c}$ whose parameters as described above satisfy the following:

\begin{items}
\item[1)]  $(\fs')^c$ has canonical generators $\{k_i',a_i',b_i'\mid i\}$ with Cartan matrix $C$,
\item[2)]  if $\sigma'(a_i')=\lambda_i' b_i'$, then  $\sign(\lambda_i')=\sign(\lambda_i)$ for all $i$,
\item[3)]  if $a_i'\in \fs'_{\varepsilon_i'}$, then $\varepsilon_i'=\varepsilon_i$ for all $i$.
\end{items}
 If all this holds, then  we can get a complex Cayley triple in ${(\fs')^c}$ as follows. Let $\varphi$ be the isomorphism of from ${\fs^c}$ to ${(\fs')^c}$ which maps $(k_i,a_i,b_i)$ to $(k_i',\mu_i a_i', \mu_i^{-1} b_i')$ where $\mu_i = \sqrt{\lambda_i/\lambda'_i}$ for all $i$. Obviously $\varphi$  is an isomorphism of $\Z$-graded Lie algebras. A short calculation shows that the antilinear homomorphisms $\varphi\circ \sigma$ and $\sigma'\circ\varphi$ agree on the canonical generators of ${\fs^c}$, thus $\varphi\circ \sigma=\sigma'\circ\varphi$. Since $\varphi$ maps the unique defining element $h/2$ of ${\fs^c}$ onto the unique defining element $h'/2$ of ${(\fs')^c}$, we have $h'=\varphi(h)$. Let $x=\varphi(e)$ and $y=\varphi(f)$, then  $(y,h',x)$ is a complex Cayley triple in ${(\fs')^c}$. Since $x\in \fs'_1$ is in general position,  $(y,h',x)$ is $(K')^c$-conjugate to $(f',h',e')$. The conclusion is that by storing the simple carrier algebras in a database we can find a
complex Cayley triple in a carrier algebra by a look-up in the database.

\section{Isomorphisms}\label{secIso}

\noindent Let $\g^c$ be a simple complex Lie algebra with real form $\g=\fk\oplus\fp$ and Cartan involution $\theta$ and complex conjugation $\sigma$. As usual, we extend $\theta$ to an automorphism of $\g^c$. Let $(\g')^c$ be a second simple complex Lie algebra with real form $\g' = \fk'\oplus \fp'$, Cartan involution $\theta'$, and complex conjugation $\sigma'$. We consider the problem to decide whether $\g$ and $\g'$ are isomorphic, and, if they are, to find an isomorphism. For this we may obviously assume that $\g^c$ and $(\g')^c$ are isomorphic.

Recall that a Cartan decomposition is unique up to conjugacy, see \cite[Theorem 5.1]{onish}. Thus, if $\g$ and $\g'$ are isomorphic, then there also exists an isomorphism $\psi\colon \g\to \g'$ with $\psi(\fk)=\fk'$ and $\psi(\fp)=\fp'$. Clearly, such an isomorphism extends to an isomorphism $\psi\colon \g^c\to (\g')^c$ with $\psi\circ\theta = \theta'\circ\psi$ and  $\psi\circ\sigma = \sigma'\circ\psi$. Conversely, if we find an isomorphism \[(\ast)\quad \psi\colon \g^c\to (\g')^c\quad\text{with } \psi\circ\theta = \theta'\circ\psi\text{ and  }\psi\circ\sigma = \sigma'\circ\psi,\] then $\psi$  restricts to an isomorphism $\psi \colon \g\to \g'$ with $\psi(\fk) = \fk'$ and $\psi(\fp) = \fp'$. 

We now describe  a construction of the isomorphism $(\ast)$, which fails if and only if $\g$ and $\g'$ are not isomorphic. Our main tool is the technique described in the following preliminary section.

\subsection{Weyl group action}\label{secWGA}
We consider the following set-up. Let $\h^c\leq \g^c$ be a Cartan subalgebra of $\g^c$ with corresponding root system $\Phi$ and basis of simple roots $\Delta=\{\alpha_1,\ldots,\alpha_\ell\}$. Let $W$ be the Weyl group associated to $\Phi$. As usual, let $\{h_i,x_i,y_i\mid i\}$  be a canonical generating set contained in a Chevalley basis $\{h_i, x_\alpha\mid i,\alpha\}$ of $\g^c$. Note that $\theta(x_\alpha) \in \g_{\alpha\circ\theta}$, and we suppose that
$\alpha\mapsto \alpha\circ\theta$ preserves $\Delta$. Then $\theta=\varphi\circ\chi=
\chi\circ\varphi$, where $\varphi$ is a diagram automorphism permuting $\Delta$, 
and $\chi$ is an inner automorphism with $\chi(h) = h$ for all $h\in \h^c$. 
Let the permutation $\pi$ be defined by 
$\varphi(\alpha_i)=\alpha_{\pi(i)}$. We further suppose that $\theta(x_\alpha)=\lambda_\alpha x_{\alpha\circ\theta}$ with $\lambda_\alpha=1$ if $\alpha\circ\theta\ne\alpha$. Thus, $\lambda_\alpha\in\{\pm 1\}$ for all $\alpha\in\Phi$; we write $\lambda_i=\lambda_{\alpha_i}$ and call $\lambda_1,\ldots,\lambda_\ell$ the parameters of $\theta$.

By abuse of notation, to $w\in W$ we associate the automorphism $w\in\Aut(\g^c)$ which maps $(h_i,x_i,y_i)$ to $(h_{w(\alpha_i)},x_{w(\alpha_i)},x_{-w(\alpha_i)})$ for all $i$. 
Let $\Delta_\theta = \{ \alpha\in \Delta \mid \alpha\circ\theta = \alpha \}$, let $\Phi_\theta$ be the root subsystem of $\Phi$ with basis $\Delta_\theta$, and let $W_\theta$ be its Weyl group. 

\begin{lemma}\label{lemWStab}
If $w=s_{\alpha_k}\in W_\theta$, then  $\alpha\to\alpha\circ \theta$   preserves the basis of simple roots $w(\Delta)$.
\end{lemma}
\begin{proof}
This follows readily if $\theta$ is inner since then  $\varphi$ is the identity and $\alpha\circ\theta = \alpha$ for all $\alpha\in \Phi$. So suppose $\varphi$ is not the 
identity, hence $\Phi$ is simply laced, cf.\ \cite[Table 1]{onish}. Note that  $\pi(k)=k$, thus  $\theta(w(x_k)) = \theta(y_k)=\lambda_k y_{k} = \lambda_k w(x_{k})$, which shows $w(\alpha_k)\circ\theta=w(\alpha_{k})\in w(\Delta)$. If $j$ is
such that  $\langle \alpha_j,\alpha_k \rangle = -1$, then $w(\alpha_j)=\alpha_j+\alpha_k$ and $w(x_j) = x_{\alpha_k+\alpha_j}$. Since $\Phi$ is simply laced,  $N_{\alpha,\beta}=\pm 1$ for all $\alpha,\beta\in\Phi$, and $[x_k,x_j] = N_{\alpha_k,\alpha_j}x_{\alpha_k+\alpha_j}$ implies that 
$\theta(w(x_j)) = \pm  x_{\alpha_k+\alpha_{\pi(j)}}$. Since also $\langle \alpha_{\pi(j)},\alpha_k 
\rangle = -1$, we have $x_{\alpha_k+\alpha_{\pi(j)}} = w(x_{\pi(j)})$, hence $w(\alpha_j)\circ\theta=w(\alpha_{\pi(j)})\in w(\Delta)$. Analogously, if $\langle \alpha_j,\alpha_k\rangle =0$, then $w(x_j)=x_j$, hence $w(\alpha_j)\circ\theta=w(\alpha_{\pi(j)})\in w(\Delta)$.
\end{proof}

Suppose $\theta$ has parameters $\lambda_1,\ldots,\lambda_\ell$, that is, $\theta(x_i)=\lambda_i x_{\pi(i)}$ for all $i$, and let $w=s_{\alpha_k}\in W_\theta$. Clearly,  $\{w(h_i),w(x_i),w(y_i)\mid i\}$ is a canonical generating set, and we modify it as follows:  Whenever $\pi(i)>i$, we replace $w(x_{\pi(i)})$ and $w(y_{\pi(i)})$ by $\theta(w(x_i))$ and $\theta(w(y_i))$; let $\{\bar h_i,\bar x_i,\bar y_i\mid i\}$ be the resulting canonical generating set with corresponding basis of simple roots $w(\Delta)$, which still is $\theta$-stable by Lemma \ref{lemWStab}. By construction, if $\pi(i)\ne i$, then $\theta(\bar x_i)=\bar x_{\pi(i)}$. Now let $\pi(j)=j$ and recall that $w(\alpha_j)=\alpha_j-\langle \alpha_j,\alpha_k\rangle \alpha_k$ and  $\pi(k)=k$. A case distinction on the value of $\langle \alpha_j,\alpha_k\rangle$ shows that \[\theta(w(x_j))=\lambda_j\lambda_k^{\langle \alpha_j,\alpha_k\rangle} w(x_j).\] In conclusion, if we replace our original  canonical generators and basis of simple roots by their (modified) images under $w\in\Aut(\g^c)$, then for the parameters 
$\tilde\lambda_j$ of  $\theta$ we have  $\tilde\lambda_j=1$ if $\pi(j)\neq j$ and 
$\tilde\lambda_j=\lambda_j\lambda_k^{\langle \alpha_j,\alpha_k\rangle}$ if $\pi(j)=j$.

\subsection{Inner type}
First we suppose that $\g$ is of inner type, that is, $\fk$ contains a Cartan subalgebra $\h$ of $\g$. Let $\Phi$ be the  root system of $\g^c$ with respect to $\h^c$, with basis of simple roots $\Delta=\{\alpha_1,\ldots,\alpha_\ell\}$. Let $\{h_i,x_i,y_i\mid i\}$ be a canonical generating set corresponding to $\Delta$. If $\g'$ is not of inner type, that is, if a Cartan subalgebra of $\fk'$ is not a
Cartan subalgebra of $\g'$, then $\g$ and $\g'$ are not isomorphic. Hence, we assume that $\g'$ is of inner type and define $\h'$, $\Phi'$, and $\Delta'$ in the same way. Since $\g^c$ and $(\g')^c$ are isomorphic we may 
assume that $\Delta$ and $\Delta'$ are ordered so that the corresponding Cartan matrices are the 
same. Recall that each root space $\g_\alpha$ with $\alpha\in \Phi$ lies either in $\fk^c$ or $\fp^c$, thus we have $\theta(x_i)=\lambda_ix_i$ with $\lambda_i\in\{\pm 1\}$ for all $i$. 
Let $\lambda_1',\ldots,\lambda_\ell'$ be defined similarly.

Suppose that we are in the situation $\lambda_i=\lambda_i'$ for all $i$, and write $\sigma(x_i)=r_i y_i$ and $\sigma(x_i')=r_i' y_i'$. By Corollary \ref{corSign}, we have  $\sign(r_i)=\sign(r_i')$ for all $i$, which allows us to define the reals $\mu_i = \sqrt{r_i/r_i'}$. Now the isomorphism $\psi\colon \g^c \to (\g')^c$ which maps $(h_i,x_i,y_i)$ to $(h_i',\mu_i x_i',\mu_i^{-1}y_i')$ for  all $i$ satisfies $\psi\circ\theta=\theta'\circ\psi$ and $\psi\circ\sigma = \sigma'\circ\psi$, and we are done.

In the remainder of this section we show how to achieve  $\lambda_i=\lambda_i'$ for all $i$ in case that $\g$ and $\g'$ are 
isomorphic. The idea is to use the results of Section \ref{secWGA} to find a new basis of simple roots such that $\theta$ and its parameters $\lambda_1,\ldots,\lambda_\ell$ are in a \emph{standard form}. As explained below, this means that at most one parameter $\lambda_k$ is negative, with certain restrictions on $k$. The possible standard forms are obtained by listing the Kac diagrams of the inner involutions of $\g^c$; to each Kac diagram corresponds exactly one standard form, and $\g$ and $\g'$ are isomorphic if and only if  the standard forms of $\theta$ and
$\theta'$ coincide.

In the following we explain this in detail for the simple Lie algebra of type $D_\ell$.

\begin{example}\label{exaDl}
Let the notation be as above and suppose $\g^c$ is of type $D_\ell$ with $\ell>4$.  We suppose that our basis of simple roots  $\Delta = \{\alpha_1,\ldots,\alpha_\ell\}$ corresponds to the labels of the following Dynkin diagram of $D_\ell$: 
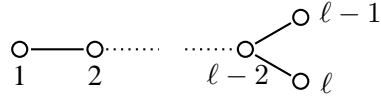
\begin{center}
  \begin{tikzpicture}[scale=0.5]\label{labDl}
    \foreach \x in {0,1,3}
    \draw[thick,xshift=\x cm] (\x cm,0) circle (2 mm);
    \draw[thick, xshift=0 cm, yshift=-7mm] (0 cm,0) node{1};
    \draw[thick, xshift=1 cm, yshift=-7mm] (1 cm,0) node{2};
    \draw[thick, xshift=3 cm, yshift=-7mm] (2.8 cm,0) node{$\ell-2$};
    \draw[thick,xshift=0 cm] (0 cm,0) ++(.3 cm, 0) -- +(14 mm,0);
    \draw[dotted, thick,xshift=1 cm] (1 cm,0) ++(.3 cm, 0) -- +(14 mm,0);
    \draw[dotted,thick] (4.4 cm,0) -- +(13mm,0);
    \draw[xshift=6 cm,thick] (30: 17 mm) circle (2mm);
    \draw[xshift=6 cm,thick] (-30: 17 mm) circle (2mm);
    \draw[thick,xshift=7.15 cm] (30: 19 mm) node{$\ell-1$};
    \draw[thick,xshift=6.6 cm] (-30: 18 mm) node{$\ell$};
    \draw[xshift=6 cm,thick] (30: 3 mm) -- (30: 14 mm);
    \draw[xshift=6 cm,thick] (-30: 3 mm) -- (-30: 14 mm);
  \end{tikzpicture}
\end{center}
Up to conjugacy, the  involutive inner automorphisms of $\g^c$ are $\bar\chi_j$ with $j=1,\ldots,\lfloor\tfrac{\ell}{2}\rfloor$ or $j=\ell-1$, where  $\bar\chi_j(x_i) = (-1)^{\delta_{ij}}x_i$ for all $i$. If we do not have that $\theta=\bar\chi_j$ for some $j$, then we proceed as follows. Recall that the parameters of $\theta$ are $\lambda_1,\ldots,\lambda_\ell$ where $\theta(x_i)=\lambda_i x_i$.  For $k=1,\ldots,\ell$ write $w_k=s_{\alpha_k}\in W$ for the reflection defined by $k$-th simple root $\alpha_k$. Let $\tilde x_i = w_k(x_i)$, $\tilde y_i = w_k(y_i)$, and
$\tilde h_i = w_k(h_i)$ be the images of the canonical generators under $w_k$. As seen in 
Section \ref{secWGA}, with respect to this new canonical generating set, $\theta$ has the same  
parameters as before, except that $\lambda_j$ is replaced by $\lambda_j\lambda_k$ if $\langle \alpha_j,
\alpha_k\rangle = -1$ (or, equivalently, if $\alpha_j$ and $\alpha_k$ are connected in the Dynkin diagram).  We will now iterate this modification of parameters. We stress that in each iteration step the reflections $w_1,\ldots,w_\ell$ are defined with respect to the new basis of simple roots constructed in the previous step; thus, acting with $w_i$ and then with $w_j$ means we first apply the reflection $s_{\alpha_i}$ and then the reflection $s_{w_i(\alpha_j)}$. Similarly, in each iteration step we have new parameters $\lambda_i$ and a new canonical generating set. By abuse of notation, in each iteration step we always denote these by the same symbols.

We now show that we can apply a sequence of simple reflections to find a new set of canonical
generators  such that for the parameters of $\theta$ there is a unique $k\in \{1,\ldots, \lfloor\tfrac{\ell}{2}\rfloor, \ell-1,\ell\}$ with $\lambda_k=-1$, that is, $\theta=\bar\chi_k$. The details are as follows:
\begin{items}
\item[$\bullet$] The first step is to achieve that at most one of $\lambda_1,\ldots,\lambda_{\ell-2}$ has value $-1$. If this is not already the case, then there exist $i<k\leq \ell-2$ with $\lambda_i,\lambda_k=-1$ and $\lambda_{j}=1$ whenever  $i<j<k$ or $k<j\leq \ell-2$.  If we act with $w_i,w_{i+1},\ldots,w_{k-1}$, then we obtain new parameters of $\theta$ with  $\lambda_{k-1}=-1$ and $\lambda_k=\ldots=\lambda_{\ell-2}=1$. Now either $k-1$ is the only index in $\{1,\ldots,\ell-2\}$ with $\lambda_{k-1}=-1$, or we iterate this process. Eventually,  at most one value of $\lambda_1,\ldots,\lambda_{\ell-2}$ is $-1$. 
\item[$\bullet$] Next, a case distinction with four cases $\lambda_{\ell-1},\lambda_\ell\in\{\pm1\}$ shows that we can in fact assume that at most one value of  $\lambda_1,\ldots,\lambda_{\ell}$ is $-1$: For example, suppose $\lambda_i=-1$ with $i<\ell-1$ and $\lambda_{\ell-1}=-1$ are the only negative parameters. If we act with $w_{\ell-1},w_{\ell-2},\ldots,w_{i+1}$, then among the new parameters the only negative ones are $\lambda_{i+1}=-1$ and $\lambda_{\ell}=-1$. By an iteration, the only negative parameters are $\lambda_{\ell-2}=\lambda_{\ell}=-1$ (or  $\lambda_{\ell-2}=\lambda_{\ell-1}=-1$), and acting with $w_\ell$ (or $w_{\ell-1}$) yields the assertion. 
\item[$\bullet$]If the only negative parameter $\lambda_k=-1$ satisfies $k\in \{1,\ldots, \lfloor\tfrac{\ell}{2}\rfloor, \ell-1,\ell\}$, then we are done. Thus, suppose we have $\lambda_k=-1$ with $\lfloor\tfrac{\ell}{2}\rfloor< k\leq \ell-2$. Let $t=\ell-k-1$, and act with $w_k,w_{k+1},\ldots,w_{k+t};\;w_{k-1},w_k,\ldots,w_{k-1+t};\;\ldots;\; w_1,w_2,\ldots,w_{1+t}$. This gives new parameters with only negative parameter $\lambda_{t+1}=-1$.
\end{items}
If $\lambda_\ell=-1$ is the only negative parameter, then we apply the diagram automorphism which fixes $\alpha_1,\ldots,\alpha_{\ell-1}$ and interchanges $\alpha_{\ell-1}$ and $\alpha_\ell$; the resulting new basis of simple roots still defines the same Cartan matrix, and now we have $\theta=\bar\chi_{\ell-1}$. 
Thus, every inner automorphisms $\theta$ of order two can be brought into \emph{standard form}, that is, there is exactly one negative parameter $\lambda_k=-1$, and  $k\in \{1,\ldots, \lfloor\tfrac{\ell}{2}\rfloor, \ell-1\}$. 
\end{example}

Our approach for the other simple Lie algebras is the same: We act with the Weyl group (as described in Section \ref{secWGA}) and certain diagram automorphisms to find a new basis of simple roots such that   $\theta$ has standard form, that is, at most one parameter $\lambda_k=-1$ is negative, with the following  restrictions: $k\leq \lceil \ell/2\rceil$ for $A_\ell$, $k=\ell$ or $k\leq \lfloor \ell/2\rfloor$ for $C_\ell$, $k=1$ for $G_2$, $k\in\{2,3\}$ for $F_4$, $k\in\{1,2\}$ for $E_6$, $k\in\{1,2,7\}$ for $E_7$, and $k\in\{1,8\}$ for $E_8$.

\begin{remark} 
A more uniform approach to the problem of finding the standard form of $\theta$ is by 
using the classification of finite order inner automorphisms as, for example, given in 
\cite{reed}. In this approach one acts with the affine Weyl group, and finding the Kac diagram of an automorphism is equivalent to finding a point in the fundamental
alcove conjugate to a given point. It can be worked out how acting by an element of the affine
Weyl group amounts to choosing a different basis of simple roots. For the purposes of 
this paper, as we are dealing with involutions only, we have chosen the more elementary 
method outlined above. 
\end{remark}

\subsection{Outer type}\label{secIsomOT}

Suppose $\theta$ is an outer involutive automorphism of $\g^c$. We apply the following four steps to $\g$ (and then $\g'$).
\begin{items}
\item[1)] The first step is to construct a $\theta$-stable Cartan subalgebra of $\g^c$: For this purpose let $\h_0^c$ be a Cartan subalgebra of $\fk^c$ and define $\h^c=C_{\g^c}(\h_0^c)$ as its centraliser in $\g^c$. It is shown in \cite[Proposition 6.60]{knapp} that $\h^c$ is a Cartan subalgebra of $\g^c$; clearly, it is fixed by $\theta$. Now $\h=\h^c\cap \g$ is a \emph{maximally compact} Cartan subalgebra of $\g$, see \cite[Proposition 6.61]{knapp}, and all Cartan subalgebras of $\g$ constructed this way are conjugate in $\g$, see \cite[Proposition 6.61]{knapp}. Thus, if $\g$ and $\g'$ are isomorphic and $\h$ and $\h'$ are Cartan subalgebras constructed as above, then there is an isomorphism $\g\to \g'$ which maps $\h$ to $\h'$.

\item[2)]The second step is to construct a basis of simple roots which is stable under $\alpha\mapsto\alpha\circ\theta$: Let $\Phi$ be the root system with respect to $\h^c$, and recall that if $\alpha\in \Phi$, then   $\beta = \alpha\circ\theta$ is a root with $\theta(x_\alpha) \in \g_\beta$ and $\theta(h_\alpha)=h_{\beta}$. This shows that the $\R$-span $\h_\R$  of all $h_\alpha$ with $\alpha\in\Phi$ is invariant under $\theta$. Moreover, $h_{0,\R}=\h_\R\cap \h_0^c$ is the 1-eigenspace of $\theta$ in $\h_\R$. Since $\h_{0,\R}$ spans $\h_0^c$ as a $\C$-vector space, the restriction of each $\alpha\in\Phi$ to $\h_{0,\R}$ is non-zero: if $\alpha(\h_{0,\R})=\{0\}$, then $\g_\alpha\subseteq C_{\g^c}(\h_0^c)=\h^c$ yields a contradiction. This shows  that there is $h_0\in \h_{0,\R}$ with $\alpha(h_0)\ne 0$ for all $\alpha\in\Phi$: such an $h_0$ can be chosen as any element outside a finite number of hyperplanes in $\h_{0,\R}$, namely, the kernels of $\alpha$ in $h_{0,\R}$. We use $h_0$ to define $\alpha>0$ if and only if $\alpha(h_0)>0$; note that elements in $\h_{\R}$ only have real eigenvalues. It is easy to check that this defines a root ordering, and, if $\alpha>0$, then also $\alpha\circ\theta>0$. Therefore the corresponding set of simples roots  $\Delta = \{\alpha_1,\ldots,\alpha_\ell\}$ is $\theta$-stable.  Let $\pi$ be the permutation of $\{1,\ldots,\ell\}$ defined by $\alpha_i\circ\theta=\alpha_{\pi(i)}$, and denote by  $\{h_i,x_i,y_i\mid i\}$ a canonical generating set corresponding to $\Delta$.

\item[3)]The third step is to adjust the coefficients of $\theta$: If $\pi(i)=i$, then set  $\tilde{h}_i=h_i$, $\tilde{x}_i=x_i$ and $\tilde{y}_i=y_i$. Otherwise, for all $(i,\pi(i))$ with $\pi(i)>i$ set  $\tilde{h}_i=h_i$, $\tilde{x}_i=x_i$, $\tilde{y}_i=y_i$,  and  $\tilde{h}_{\pi(i)}=\theta(h_i)$, $\tilde{x}_{\pi(i)}=\theta(x_i)$, $\tilde{y}_{\pi(i)}=\theta(y_i)$. By replacing  $\{h_i,x_i,y_i\mid i\}$  with the canonical generating set $\{\tilde{h}_i,\tilde{x}_i,\tilde{y}_i\mid i\}$, cf.\ Proposition \ref{prop:cangens}, we may assume that $\theta(x_i)=\lambda_i x_{\pi(i)}$  with $\lambda_i=1$ if $\pi(i)\neq i$.

\item[4)]Finally, we decompose $\theta$: Let $\varphi$ be the diagram automorphism defined by $\pi$ with respect to the new canonical generating set defined in 3), cf.\ Section \ref{secRF}, that is, $\varphi(x_i)=x_{\pi(i)}$, $\varphi(y_i)=y_{\pi(i)}$, and $\varphi(h_i)=h_{\pi(i)}$ for all $i$. By construction, \[\chi=\varphi\circ \theta=\theta\circ\varphi\] is an involutive inner automorphism of $\g^c$ with  $\chi(x_i)=x_i$ if $\pi(i)\ne i$, and   $\chi(x_i)=\lambda_i x_i$ if $\pi(i)=i$ and $\theta(x_i)=\lambda_i x_i$; clearly, $\lambda_i=\pm 1$. The analogous statement holds for $y_i$.
\end{items}
We use the same procedure to construct a $\theta'$-stable set of positive roots $\Delta'$, and automorphisms $\varphi'$ and $\chi'$ of $(\g')^c$. We also assume that the bases $\Delta$ and $\Delta'$ are ordered such that the corresponding Cartan matrices are the same, and $\pi=\pi'$ as permutations of $\{1,\ldots,\ell\}$. If the latter is not possible, then $\g^c$ and $(\g')^c$ are not isomorphic.  Let $\{h_i,x_i,y_i\mid i\}$ and $\{h_i',x_i',y_i'\mid i\}$ be the associated sets of canonical generators as constructed in Step 3) above, and let $\psi\colon \g^c\to (\g')^c$ be the associated isomorphism. We now try to modify $\psi$ so that it is compatible with $\theta,\theta'$ and $\sigma,\sigma'$.\medskip

\subsubsection{Make $\psi$ compatible with $\theta$ and $\theta'$} Recall that  $\theta=\chi\circ \varphi$, $\theta'=\chi'\circ\varphi'$, and $\psi\circ \varphi = \varphi'\circ \psi$. If $\pi(i)\ne i$, then \[\theta'\circ \psi(x_i)=\theta'(x'_i)=x'_{\pi'(i)}=x'_{\pi(i)}=\psi(x_{\pi(i)})=\psi\circ \theta(x_i);\] similarly, $\theta'\circ \psi$ and $\psi\circ \theta$ coincide on the whole subspace of $\g^c$ spanned by all $x_i, y_i, h_i$ with $\pi(i)\ne i$. If $\pi(i)=i$ with $\theta(x_i)=\lambda_i x_i$ and $\theta'(x'_i)=\lambda'_i x'_i$, then $\theta'\circ\psi(x_i)=\psi\circ \theta(x_i)$ if and only if $\lambda_i=\lambda_i'$. 

\begin{items}
\bullit Type $A_\ell$: If $\ell=2m$ is even, then $\pi$ acts fixed point freely, thus  $\psi$ as constructed above already satisfies $\theta'\circ \psi=\psi\circ \theta$. If $\ell=2m+1$, then $\pi$ has exactly one fixed point, say $i=1$, and either $\lambda_{1}=1$ or $\lambda_{1}=-1$. On the other hand, up to conjugacy, $A_\ell$ has two outer automorphisms, so each choice for $\lambda_1$ corresponds to a different conjugacy class of automorphisms. Thus, if $\g^c$ and $(\g')^c$ are isomorphic, then $\lambda_{1}=\lambda_{1}'$, and $\psi$ is an isomorphism with $\theta'\circ \psi=\psi\circ \theta$. 
\bullit Type $E_6$:
Here $\pi$ has two fixed points, say $i=2,4$, thus there are four possible combinations of signs for $\lambda_2$ and $\lambda_4$. However, up to conjugacy, $E_6$ has two outer automorphisms. Suppose our root basis $\Delta=\{\alpha_1,\ldots,\alpha_6\}$ corresponds to the labels of the following Dynkin diagram of $E_6$:

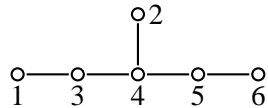
\begin{center}
  \begin{tikzpicture}[scale=.4]
    \foreach \x in {0,...,4}
    \draw[thick,xshift=\x cm] (\x cm,0) circle (2 mm);
   
    \draw[thick, xshift=0 cm, yshift=-7mm] (0 cm,0) node{1};
    \draw[thick, xshift=1 cm, yshift=-7mm] (1 cm,0) node{3};
    \draw[thick, xshift=2 cm, yshift=-7mm] (2 cm,0) node{4};
    \draw[thick, xshift=3 cm, yshift=-7mm] (3 cm,0) node{5};
    \draw[thick, xshift=4 cm, yshift=-7mm] (4 cm,0) node{6};
    \foreach \y in {0,...,3}
    \draw[thick,xshift=\y cm] (\y cm,0) ++(.3 cm, 0) -- +(14 mm,0);
    \draw[thick] (4 cm,2 cm) circle (2 mm);
    \draw[thick] (4.6 cm,2 cm) node{2};
    \draw[thick] (4 cm, 3mm) -- +(0, 1.4 cm);
  \end{tikzpicture}
\end{center}
Up to conjugacy, $E_6$ has two outer automorphisms $\varphi\circ\bar\chi$, where $\varphi$ is the diagram automorphism acting via the permutation $\pi=(1,6)(3,5)$, and $\bar\chi$ is an inner automorphisms which satisfies $\bar\chi(x_4)=\pm x_4$ and $\bar\chi(x_i)=x_i$ if $i\ne 4$. 
As outlined in Section \ref{secWGA} we now act with $w_2=s_{\alpha_2}$ and $w_4=s_{\alpha_4}$ in order
to find a new canonical generating set (with respect to  a new basis of simple roots), relative 
to which we have $\lambda_2=\lambda_4=1$, or $\lambda_2=1$ and $\lambda_4 = -1$. It is 
straightforward to see that this can always be done. For example, if $\lambda_2=-1$ and
$\lambda_4=1$, then we first act with $w_2$ to get $\lambda_2=\lambda_4=-1$ and subsequently
with $w_4$ to get $\lambda_2=1$ and $\lambda_4 = -1$. Finally we use the same trick as
in the beginning of Section \ref{secIsomOT} to obtain $\lambda_i=1$ 
for all $i\ne 2,4$ (that is, we set $x_5 = \theta(x_3)$ etc.). 
The conclusion is that we can arrange  that 
$\lambda_i=\lambda_i'$ for every $i$, hence $\psi$ is  an isomorphism with 
$\theta'\circ \psi=\psi\circ \theta$. 

\bullit Type $D_\ell$: We proceed as for $E_6$ and suppose that our basis of simple roots $\Delta$ corresponds to the Dynkin diagram of $D_\ell$ as shown on page \pageref{labDl}. Up to conjugacy, the  involutive outer automorphisms of $D_\ell$ are $\varphi\circ \bar\chi$, where $\varphi$ is the diagram automorphism defined by $\pi=(\ell-1,\ell)$, and $\bar\chi$ is an
inner automorphism with $\bar\chi(x_i) = \bar\lambda_i x_i$, where either $\bar\lambda_i=1$ for all $i$, or there exists a unique negative $\bar\lambda_k$ and $k\in\{1,\ldots, \lceil\tfrac{\ell}{2}\rceil-1\}$. As in Example \ref{exaDl}, we act with reflections $s_{\alpha_j}\in W$,  $j\in\{1,\ldots,\ell-2\}$, to find a canonical generating set relative to which there is a unique negative parameter $\lambda_k$, and $k\in \{1,\ldots,\ell-2\}$. If  $k\leq \lceil \tfrac{\ell}{2} \rceil -1$, then we are done; otherwise we proceed as follows. Set $\beta_i = \alpha_{\ell-i-1}$ for $i=1,\ldots,\ell-2$, and $\beta_{\ell-1} = -\alpha_1-
\cdots - \alpha_{\ell-2}-\alpha_{\ell-1}$,  and $\beta_{\ell} = -\alpha_1-
\cdots - \alpha_{\ell-2}-\alpha_{\ell}$. Then $\bar\Delta = \{\beta_1,\ldots,\beta_\ell\}$ is also
a basis of simple roots with the same Dynkin diagram. Now we take a canonical generating set
with respect to $\bar\Delta$.
With respect to this new canonical generating set, $\chi$ has a unique negative parameter 
$\lambda_k=-1$, and $k\in\{1,\ldots,\lceil \tfrac{\ell}{2}\rceil -1\}$. 
\end{items}

\noindent Using these constructions, we can arrange that  $\lambda_i=\lambda_i'$ for all $i$, thus the corresponding isomorphism $\psi$ (defined on the newly constructed canonical generating sets) is compatible with $\theta$ and $\theta'$.

\medskip

\subsubsection{Make $\psi$ compatible with $\sigma$ and $\sigma'$}
Using the construction in the previous paragraphs, we have established that  either $\g$ and 
$\g'$ are not isomorphic, or we have an isomorphism $\psi\colon \g^c\to (\g')^c$
compatible with $\theta$ and $\theta'$. We assume the latter holds, and we now  adjust 
$\psi$ so it is also compatible with the complex conjugations $\sigma$ and $\sigma'$; this  yields the desired isomorphism between $\g$ and $\g'$.

By our previous construction, if $i\ne \pi(i)$, then $\theta(x_i)=x_{\pi(i)}$, and $\theta(x_i)=\lambda_i x_i$ otherwise. Lemma \ref{lemSigma} shows that $\sigma(x_i)=r_i y_{\pi(i)}$ for some $r_i\in\R$ with $r_i=r_{\pi(i)}$. If $i\ne \pi(i)$, then $r_i<0$, and, if $i=\pi(i)$, then $r_i$ and $-\lambda_i$ have the same sign, see Corollary \ref{corSign}.  Now define  define $\mu_i = \sqrt{1/ |r_i|}$ for $i=1,\ldots,\ell$. If we replace $x_i,y_i,x_{\pi(i)},y_{\pi(i)}$ by  $\tilde x_i = \mu_i x_i$, $\tilde y_i = \mu_i^{-1} y_i$,
$\tilde x_{\pi(i)} = \mu_i x_{\pi(i)}$, $\tilde y_{\pi(i)} = \mu_i^{-1} y_{\pi(i)}$, then  we get a new set of canonical
generators where $\theta$ acts in the same way and $\sigma(\tilde x_i) = \pm \tilde y_{\pi(i)}$ for all $i$. In a similar way, we obtain a new set of canonical generators $\{\tilde x_i',\tilde y_i',h_i'\mid i\}$ of $(\g')^c$; recall that $\lambda_i=\lambda_i'$ for all $i$. The associated isomorphism $\g^c\to (\g')^c$ now is compatible with $\theta$, $\sigma$, and $\theta'$, $\sigma'$, and restricts to an isomorphism $\g\to \g'$ preserving the Cartan decompositions.

\begin{remark}
In the algorithms described in this section we compute  root systems of $\g^c$ and $(\g')^c$ with respect to
Cartan subalgebras $\h^c$ and $(\h')^c$. In order for that to work well we
need Cartan subalgebras that split over $\SqrtField(\imath)$ (or an extension thereof
of small degree). However, the problem of finding such Cartan subalgebras is very
difficult, cf.\ \cite{irs}. Therefore, in our algorithms we assume that we have a 
Cartan subalgebra with a small splitting field. 
\end{remark}

\subsection{Nilpotent orbits under isomorphisms}
Suppose $\g=\fk\oplus\fp$ and $\g'=\fk'\oplus\fp'$ are semisimple real Lie algebras and $\psi\colon \g\to \g'$ is an isomorphism such that $\psi(\fk)=\fk'$ and $\psi(\fp)=\fp'$. As described in the previous sections, we can extend this to an isomorphism $\psi\colon \g^c\to (\g)^c$ compatible with the corresponding Cartan involutions $\theta,\theta'$ and complex conjugations $\sigma,\sigma'$. Let $G$  be the connected Lie subgroup of the adjoint group $G^c$ of $\g^c$, having Lie algebra $\g$. Similarly, let $G'$ be defined for $\g'$.

\begin{lemma}
The isomorphism $\psi\colon \g\to \g'$ maps nilpotent orbits to nilpotent orbits.
\end{lemma}
\begin{proof}
Clearly, $e\in\g$ is nilpotent if and only if $\psi(e)\in\g'$ is nilpotent. We show that if two nilpotent  $e,f\in\g$ are conjugate under $G$, then $e'=\psi(e)$ and $f'=\psi(f)$ are conjugate under $G'$; then the same argument with $\psi$ replaced by $\psi^{-1}$ proves the assertion. As shown in \cite[pp.\ 126--127]{helgason}, the adjoint group $G$ is generated by all $\exp\ad x$ with $x\in \g$, and the isomorphism $\psi$ lifts to an isomorphism $\tilde\psi\colon G\to G'$, $\beta\mapsto \psi\circ\beta\circ\psi^{-1}$. Thus, if  $\beta(e)=f$ for some $\beta\in G$, then $\tilde\psi(\beta)(\psi(e))=\psi(f)$, and $\psi(e)$ and $\psi(f)$ are $G'$-conjugate in $\g'$.
\end{proof}

\begin{appendix}

\section{Comment on the implementation}\label{secSF}
For computing with semisimple Lie algebras we use the package {\tt SLA} \cite{sla} for the computer algebra system {\tt GAP} \cite{gap}. This package provides the functionality, for example,  to computer Chevalley bases, canonical generators, and involutive automorphisms. In Section \ref{secNonPrincCase} we use the Gr\"obner bases functionality of the computer algebra system {\sc Singular} \cite{singular} via the linkage package  {\tt Singular} \cite{singularGap}.

\subsection{The field $\SqrtField$}
\noindent We now comment on the field $\SqrtField=\Q(\{\sqrt{p}\mid p\textrm{ a prime}\})$.   {\tt GAP}  already allows us to work with subfields of cyclotomic fields $\Q(\zeta_n)$, where $\zeta_n$ is a complex primitive $n$-th root of unity. However, if $x=\sum_{i=1}^m\sqrt{p_i}$ for primes $p_1,\ldots,p_m$, then the smallest $n$ with $x\in\Q(\zeta_n)$ is $n=\lcm(e_1,\ldots,e_m)$ where $e_i=p_i$ if $p_i\equiv 1\bmod 4$, and $e_i=4p_i$ otherwise, cf.\ Lemma \ref{lemPCyc}. Thus, already for small $m$ this requires to work in large cyclotomic fields. Alternatively, one could work in an algebraic extension defined by an irreducible polynomial over $\Q$. The disadvantage here is that we do not know in the beginning which irrationals  turn up, so we would have to extend and therefore change the underlying field several times. To avoid all this, we have implemented our own realisation of $\SqrtField(\imath)$. Every element of $\SqrtField(\imath)$ can be written uniquely as $u=\sum_{i=1}^n r_i \sqrt{z_i}$ where $z_i>0$ are pairwise distinct squarefree integers and $r_i\in\Q(\imath)$. Internally, we represent $u$ as a list with entries $(r_i,z_i)$, which allows efficient addition and multiplication in $\SqrtField(\imath)$. A computational bottleneck is the construction of the multiplicative inverse of such an $u\ne 0$: We compute powers $\{1,u,u^2,\ldots,u^m\}$ until $u^m$ can be expressed as a $\Q$-linear combination of $\{1,u,\ldots,u^{m-1}\}$, say $u^m=\sum_{i=0}^{m-1} q_i u^i$. The minimal polynomial of $u$ over $\Q$ is $f(x)=x^m-\sum_{i=0}^{m-1}q_ix^i=xg(x)+q_0$, therefore $u^{-1}=-g(u)/q_0$. Although all this can done with linear algebra,  $m$ can become rather large, cf.\ Lemma \ref{lemAllKi}.

Often we had to deal with the following problem: Suppose $v\in\SqrtField(\imath)$ is given as an element of  $\Q(\zeta_n)$ for some $n$,  write it as an element of $\SqrtField(\imath)$, that is, $v=\sum_{i=1}^m r_i\sqrt{k_i}$ for pairwise distinct positive squarefree integers $k_i$ and Gaussian rationals $r_i$.  Clearly, it is sufficient to consider $v$ real. The first step is to determine the set $\mathcal{S}$ of all positive squarefree $k$ with $\sqrt{k}\in\Q(\zeta_n)$; we do this in Corollary \ref{lemSQCyc}. The second step is to  make the ansatz $v=\sum_{k\in\mathcal{S}} r_k \sqrt{k}$ in $\Q(\zeta_n)$ with indeterminates $r_k\in\Q$. Linear algebra can be used to find a solution of this equation; we prove in Lemma \ref{lemAllKi} that such a solution always exists.  We now provide the theoretical background of this approach; our starting point is the following lemma, see \cite[p.\ 56, Proposition 3]{shirali}, and its corollary. 

\begin{lemma}\label{lemPCyc}
If $p$ is an odd prime, then $\sqrt{(-1)^{(p-1)/2}p}\in \Q(\zeta_p)$. 
\end{lemma}

\begin{corollary}\label{lemSQCyc}Let $k$ and $n$ be positive integers. Suppose $k$ is squarefree and let $e$ be the number of primes $p\equiv 3\bmod 4$ dividing $k$.
\begin{ithm}
\item If $\sqrt{2}\in\Q(\zeta_n)$, then $8\mid n$. If $\sqrt{k}\in\Q(\zeta_n)$, then $k\mid n$.
\item If $n$ is odd, then $\Q(\zeta_n)=\Q(\zeta_{2n})$, and $\sqrt{k}\in\Q(\zeta_n)$ if and only if $e$ is even and $k\mid n$.
\item If $4\mid n$ and $8\nmid n$, then $\sqrt{k}\in\Q(\zeta_n)$ if and only if $k$ is odd and $k\mid n$.
\item If $8\mid n$, then $\sqrt{k}\in Q(\zeta_n)$ if and only if $k\mid n$.
\item Let $n$ be minimal with $\sqrt{k}\in\Q(\zeta_n)$. If $k$ is odd and $e$ is even, then  $n=k$, and $n=4k$ otherwise.
\end{ithm}
\end{corollary}

\begin{lemma}\label{lemHelpAllKi2}\label{lemAllKi}
\begin{ithm}
\item Let $n,k_1,\ldots,k_m$ be pairwise distinct positive squarefree integers and suppose there exists a prime $p\mid n$ with $p\nmid k_i$ for all $i$. Then $\sqrt{n}\notin\Q(\sqrt{k_1},\ldots,\sqrt{k_m})$.
\item Let $v=\sum_{i=1}^m r_i \sqrt{k_i}\in\SqrtField$ for rational $r_i\ne 0$ and pairwise distinct positive squarefree integers $k_i$. Then $v$ is a primitive element of $\Q(\sqrt{k_1},\ldots,\sqrt{k_m})$.
\end{ithm}
\end{lemma}
\begin{proof}
\begin{iprf}
\item We use induction on $m$. The assertion clearly holds if  $m=1$, thus let $m>1$ and write  $K'=\Q(\sqrt{k_1},\ldots,\sqrt{k_{m-1}})$ and $K=K'(b)$ with  $b=\sqrt{k_m}$. Suppose that $\sqrt{n}\in K$. Since $\sqrt{n}\notin K'$ by the induction hypothesis, $b\notin K'$ and, therefore,  $\sqrt{n}=r+bs$ for unique $r,s\in K'$. Note that $s,r\ne 0$ since otherwise  $\sqrt{n}$ or $\sqrt{nk_m}/k_m$ would lie in $K'$, a contradiction.  Now squaring yields $b=(n-r^2-s^2k_m)/(2rs)\in K'$, the final contradiction.
\item Suppose $K=\Q(\sqrt{k_1},\ldots,\sqrt{k_m})=\Q(\sqrt{k_1},\ldots,\sqrt{k_s})$ has degree $d=2^s$  over $\Q$ with $s\leq m$. Since $K$ is the splitting field of the separable polynomial $(x^2-k_1)\ldots(x^2-k_s)$, the extension is Galois and therefore $\mathcal{G}=\textrm{Gal}(K/\Q)$ has order $d$. Clearly, every map defined by $\sqrt{k_i}\mapsto \pm\sqrt{k_i}$ for $i=1,\ldots,s$ gives rise to a Galois automorphism, and an order argument shows that  $\mathcal{G}$ consists exactly of these automorphisms. We now show that $1,\sqrt{k_1},\ldots,\sqrt{k_s}$ are linearly independent over $\Q$.  Clearly, this is true for $s=1$, so let $s\geq 2$. For a contradiction, assume $(\dagger)$ $\sum\nolimits_{i=1}^s r_i \sqrt{k_i}+r_{m+1}=0$ for rationals $r_i$. Let $p$ be a prime dividing $k_1\ldots k_s$. Now  $(\dagger)$  implies that $\sqrt{p}$ lies in the field generated by $\sqrt{k_1'},\ldots,\sqrt{k_s'}$ with $k_i'=k_i/\gcd(k_i,p)$, contradicting part a). Let $f$ be the minimal polynomial of $v$ over $\Q$. Clearly, $\gamma(v)$ is a root of $f$ for every $\gamma\in\mathcal{G}$. Since $\sqrt{k_1},\ldots,\sqrt{k_s}$ are $\Q$-linearly independent, it follows that  $\gamma(v)\ne \gamma'(v)$ for all $\gamma\ne\gamma'$ in $\mathcal{G}$. This shows that $f$ has at least $d$ different zeros, which implies that $f$ has in fact degree $d$ and  $v$ is primitive.
\end{iprf}
\end{proof}
\end{appendix}

\begin{small}
\bibliographystyle{amsplain}

\begin{thebibliography}{10} 

\bibitem{bourbaki}
N.\ Bourbaki. 
Groupes et alg\`ebres de Lie, Chapitres VII et VIII, Hermann, Paris, 1975.

\bibitem{CM}
D.\ H.\ Collingwood and W.\ M.\ McGovern. Nilpotent orbits in semisimple Lie algebras. Van Nostrand Reinhold Mathematics Series. Van Nostrand Reinhold Co., New York, 1993.


\bibitem{singularGap}
 M.\ Costantini and  W.\ A.\ de Graaf.
{\tt Singular} -- A GAP4 package, 2006.

\bibitem{clo}
D.\ Cox, J.\ Little, and D.\ O'Shea.
\newblock Ideals, Varieties and Algorithms: An Introduction to
  Computational Algebraic Geometry and Commutative Algebra.
\newblock Springer Verlag, New York, Heidelberg, Berlin, 1992.


\bibitem{singular}
W.\ Decker, G.-M.\ Greuel, G.\ Pfister, and M.\ Sch{\"o}nemann: 
\newblock {\sc Singular} {3-1-3} -- {A} computer algebra system for polynomial computations.
\newblock {http://www.singular.uni-kl.de} (2011).

 \bibitem{djok87}
 D.\ \v{Z}.\ \DJ okovic. 
\emph{Proof of a Conjecture of Kostant}, 
Trans.\ AMS {\bf 302} (2) (1987), 577-585.

\bibitem{djokG2}
 D.\ \v{Z}.\ \DJ okovic. 
\emph{Explicit Cayley triples in real forms of G2,F4 and E6}, 
Pacific J.\ Math.\ {\bf 184} (2) (1998), 231-255.

\bibitem{djokE7}
 D.\ \v{Z}.\ \DJ okovic. 
\emph{Explicit Cayley triples in real forms of E7},
Pacific J.\ Math.\ {\bf 191} (1) (1999), 1-23.

\bibitem{djokE8}
 D.\ \v{Z}.\ \DJ okovic. 
\emph{Explicit Cayley triples in real forms of E8}, 
Pacific J.\ Math.\ {\bf 194} (1) (2000), 57-82.

\bibitem{djokInner}
D.\ \v{Z}.\ \DJ okovic. 
\emph{Classification of nilpotent elements in simple exceptional real Lie algebras of inner type and description of their centralizers}, 
J.\ Algebra {\bf 112} (2) (1988), 503-524. 

\bibitem{galina}
E.\ Galina.
\emph{Weighted Vogan diagrams associated to real nilpotent orbits},
New developments in Lie theory and geometry, 239-253, Contemp.\ Math., 491, AMS, Providence, RI, 2009. 


\bibitem{gap}
The GAP Group, {\tt GAP} -- Groups, Algorithms, and Programming, Version 4.5.5.
www.gap-system.org (2012).

\bibitem{gorb}
V.\ V.\ Gorbatsevich, A.\ L.\ Onishchik and  \`E.\ B.\ Vinberg. 
Lie groups and Lie algebras III. Springer, 1994.

\bibitem{deGraafBook} W.\ A.\ de Graaf. 
Lie Algebras: Theory and Algorithms.
vol.\ 56 of North-Holland Math.\ Lib. Elsevier Science, 2000.

\bibitem{sla} W.\ A.\ de Graaf. 
{\tt SLA} -- computing with Simple Lie Algebras. A GAP4 package.\newline
 www.science.unitn.it/$\sim$degraaf/sla.html (2012).

\bibitem{deGraafTheta} W.\ A.\ de Graaf. 
\emph{Computing representatives of nilpotent orbits of $\theta$-groups}, 
J.\ Symb.\ Comput.\ {\bf 46} (2011), 438-458. 

\bibitem{helgason} S.\ Helgason.
Differential Geometry, Lie Groups, and Symmetric Spaces.
Academic Press, New York San Francisko London, 1978. 

 
\bibitem{humph} J.\ E.\ Humphreys. 
Introduction to Lie algebras and representation theory. 
Second printing, revised. Graduate Texts in Mathematics, 9. Springer-Verlag, New York-Berlin, 1978

 
\bibitem{humph2} J.\ E.\ Humphreys. Reflection groups and Coxeter groups. Cambridge University Press, Cambridge, 1990.

\bibitem{irs}
G.\ Ivanyos, L.\ Ronyai and J.\ Schicho,
\emph{Splitting full matrix algebras over algebraic number fields}, 
J.\  Algebra {\bf 354} (2012), 211-223.

\bibitem{jac}
N.\ Jacobson. Lie Algebras. New York-London: Wiley Interscience, 1962.

\bibitem{kac}
V.\ G.\ Kac. Infinite dimensional Lie algebras. Third edition. Cambridge University Press, 1990.

\bibitem{knapp}
A.\ W.\ Knapp. Lie groups beyond an introduction. Second edition. Progress in Mathematics, 140. Birkh\"auser, 2002.

\bibitem{littelmann}
P.\ Littelmann.
An effective method to classify nilpotent orbits. In {\it Algorithms in
algebraic geometry and applications (Santander, 1994), Progr.\ Math.},
{\bf 143}, 255-269, Birkh\"auser, Basel, 1996.


\bibitem{rallis}
B.\ Kostant and B.\ Rallis.
\emph{Orbits and representations associated with symmetric spaces}, 
Amer.\ J.\ Math.\ {\bf 93} (1971), 753-809. 

\bibitem{noela}
A.\ No\"{e}l.
\emph{Classification of admissible nilpotent orbits in simple exceptional real Lie algebras of inner type}, 
Represent.\ Theory {\bf 5} (2001), 455-493.

\bibitem{noelb}
A.\ No\"{e}l.
\emph{Classification of admissible nilpotent orbits in simple real Lie algebras E6(6) and E6(-26)},
Represent.\ Theory {\bf 5} (2001), 494-502.

\bibitem{noel}
A.\ No\"{e}l.
\emph{Nilpotent orbits and theta-stable parabolic subalgebras}, 
Represent.\ Theory {\bf 2} (1998), 1-32.


\bibitem{onish} A.\ L.\ Onishchik. 
Lectures on Real Semisimple Lie Algebras and Their Representations.  
ESI Lectures in Mathematics and Physics. European Mathematical Society (EMS), Z\"urich, 2004.

\bibitem{reed}  M.\ Reeder,
\emph{Torsion automorphisms of simple Lie algebras},
L'Enseignement Mathematique (2), {\bf 56}, (2010), 3-47.


\bibitem{rot72} L.\ P.\ Rothschild. 
\emph{Orbits in a real reductive Lie algebra}, 
Trans.\ AMS {\bf 168} (1972), 403-421. 

\bibitem{seki87}
J.\ Sekiguchi. 
\emph{Remarks on real nilpotent orbits of a symmetric pair}, 
J.\ Math.\ Soc.\ Japan {\bf 39} (1) (1987), 127-138.

\bibitem{shirali}
S.\ Shirali.
Number Theory
Echoes from Resonance Series, Universities Press, 2004


\bibitem{vin87}
\`E.\ B.\ Vinberg. 
\emph{Classification of Homogeneous Nilpotent Elements of a Semisimple Graded Lie Algebra}, 
Trudy Sem.\ Vektor.\ Tenzor.\ Anal.\ {\bf 19} (1979), 155-177. English trans.: Selecta Mathematica Sovietica {\bf 6} (1) (1987), 15-35.

\end{thebibliography}

\end{small}
\end{document}